\documentclass[a4paper,12pt]{amsart}
\usepackage{ucs} 
\usepackage[utf8x]{inputenc}
\usepackage[english]{babel}
\usepackage{amsmath}
\usepackage{amsfonts}
\usepackage{amssymb}
\usepackage{hyperref}
\usepackage{amsxtra,latexsym,amscd,amsthm,marvosym,multirow}
\usepackage[a4paper,inner=2cm, outer=2cm, top=2cm, bottom=2cm]{geometry}
\usepackage{setspace}
\usepackage[pdftex]{graphicx}
\usepackage{framed, fancybox}
\usepackage{epstopdf}
\usepackage{enumerate}
\usepackage{csquotes}
\usepackage{mathrsfs}
\usepackage{xcolor}

\newtheorem{theorem}{Theorem}[section]
\newtheorem{lemma}{Lemma}[section]
\newtheorem{proposition}[theorem]{Proposition}
\newtheorem{corollary}[theorem]{Corollary}
\newtheorem{assumption}{Assumption}[section]
\theoremstyle{remark}
\newtheorem{definition}[theorem]{Definition}
\newtheorem{notation}{Notation}[section]
\newtheorem{remark}{Remark}[section]

\newcommand{\RR}{\mathbb{R}}

\newcommand{\ZZ}{\mathbb{Z}}
\newcommand{\NN}{\mathbb{N}}

\newcommand{\A}{\textbf{A}}
\newcommand{\B}{\textbf{B}}
\newcommand{\dd}{\textup{d}}
\newcommand{\V}{\mathcal{V}}

\makeatletter

\@addtoreset{equation}{section}  
\makeatother

\begin{document}
\title{Magnetic WKB constructions on surfaces}

\author{Y. Guedes Bonthonneau} \address[Y. Guedes Bonthonneau]{Univ
  Rennes, CNRS, IRMAR - UMR 6625, F-35000 Rennes, France}
\email{yannick.bonthonneau@univ-rennes1.fr}

\author{T. Nguyen Duc} \address[T. Nguyen Duc]{Univ Rennes, CNRS,
  IRMAR - UMR 6625, F-35000 Rennes, France}
\email{duc-tho.nguyen@univ-rennes1.fr}

\author{N. Raymond} \address[N. Raymond]{Laboratoire Angevin de
  Recherche en Mathématiques, LAREMA, UMR 6093, UNIV Angers, SFR
  Math-STIC, 2 boulevard Lavoisier 49045 Angers Cedex 01, France}
\email{nicolas.raymond@univ-angers.fr}

\author{S. V\~u Ng\d{o}c} \address[S. V\~u Ng\d{o}c]{Univ Rennes,
  CNRS, IRMAR - UMR 6625, F-35000 Rennes, France}
\email{san.vu-ngoc@univ-rennes1.fr}
\maketitle

\begin{abstract}
  This article is devoted to the description of the eigenvalues and
  eigenfunctions of the magnetic Laplacian in the semiclassical limit
  via the complex WKB method. Under the assumption that the magnetic
  field has a unique and non-degenerate minimum, we construct the
  local complex WKB approximations for eigenfunctions on a general
  surface. Furthermore, in the case of the Euclidean plane, with a
  radially symmetric magnetic field, the eigenfunctions are
  approximated in an exponentially weighted space.
\end{abstract}




\section{Introduction}
\subsection{Magnetic Laplacian on a Riemannian manifold}
\ Let $(M,g)$ be a two-dimensional connected oriented Riemannian
manifold (possibly with boundary) equipped with a metric $g$. Let $\textbf{A}$ be a smooth real-valued $1$-form defined on
$M$. Since $M$ is two-dimensional, there exists a smooth real-valued
function $\B$ such that
\[ \dd \textbf{A} = \textbf{B}\, \dd \V_{g}\,,\] in which $\dd \V_{g}$
is the Riemannian volume form on $M$. We call $\textbf{A}$ the
magnetic potential and $\textbf{B}$ the magnetic field.  Let
$\hat{g} : T M \to T^{*}M$ be the canonical isomorphism induced by
$g$, \emph{i.e.} for each $p\in M$ and $V\in T_{p} M$,
\begin{equation}
  \hat{g}_{p}(V)(W) = g_{p}(V,W) \qquad \text{ for all } W\in T_{p}M \,.
\end{equation}
Let $g^{*}$ be the dual metric on the set of $1$-forms on $M$:
\[
  g_{p}^{*}(w_{1} ,w_{2}) = g_{p}(\hat{g}_{p}^{-1}
  (w_1),\hat{g}_{p}^{-1} (w_2)) \qquad \text{ for all } w_{1} ,w_{2}
  \in T_{p}^{*}M \,.
\]
The magnetic Laplacian can de defined as follows.  When $M$ is a
compact manifold, possibly with boundary, consider the sesquilinear
form defined for complex-valued functions $u,v \in H_{0}^1(M)$ by
\[ Q_{h,\textbf{A}}(u,v) = \int_M g^{*}\left((-ih\dd
    -\textbf{A})u,(-ih\dd -\textbf{A})v\right)\, \dd \V_{g}\,.\] From
the Lax-Milgram Theorem, $Q_{h,\textbf{A}}$ is associated with a
self-adjoint operator $\mathscr{L}_{h,\textbf{A}}$ whose domain is
given by
\begin{align*}
  \textup{Dom}(\mathscr{L}_{h,\textbf{A}})
  & = \left\{\begin{aligned}
      u \in H_{0}^1(M): &
      \text{ there exists } f \in \mathrm{L}^{2}(M)  \text{ such that } \\
      & Q_{h,\textbf{A}}(u,v) = \langle f,v \rangle_{\mathrm{L}^{2}(M)}
      \qquad \text{for all } v \in H_{0}^1(M)
    \end{aligned}\right\}\\
  & =\mathrm{H}_{0}^{1}(M) \cap
    \mathrm{H}^{2}(M)\,,
\end{align*}
and
\[ \left\langle \mathscr{L}_{h,\textbf{A}} u ,
    v\right\rangle_{\mathrm{L}^{2}(M)}= Q_{h,\textbf{A}}(u,v) \qquad
  \forall u\in \textup{Dom}(\mathscr{L}_{h,\textbf{A}}), \forall v \in
  H_{0}^1(M)\,.\] Choosing local coordinates $(x^1,x^2)$ on $M$, the operator
$\mathscr{L}_{h,\textbf{A}}$ can be written explicitly as
\begin{equation}\label{EQ ML on Manifold}
  \mathscr{L}_{h,\textbf{A}}= \frac{1}{\sqrt{\vert G \vert}} \sum_{k,\ell=1}^{2} \left(h\mathrm{D}_{k}-A_{k} \right) \left[\sqrt{\vert G \vert}G^{k\ell} \left(h\mathrm{D}_{\ell}-A_{\ell} \right) \right]\,,
\end{equation}
where $\mathrm{D}_{j}:= -i\frac{\partial}{\partial x^{j}} $, $G$ is
the matrix associated with $g$, and $G^{k\ell}$ are the matrix
elements of $G^{-1}$.

\begin{remark}
When $M$ is compact, we have $H_{0}^{1}(M)=H^{1}(M)$ and
  \begin{equation*}
    \textup{Dom}(\mathscr{L}_{h,\textbf{A}}) = H^{2}(M)\,.
  \end{equation*}
  The readers may consult \cite{HM96,Shubin01,HK11} for an
  introduction to the magnetic Laplacian on a Riemannian manifold. We
  also refer the reader to the book \cite{Raymond17} for the study of
  the magnetic Laplacian under various aspects.
\end{remark}
In this paper, we will also consider the non-compact case $M =\RR^2$
with the Euclidean metric. In this case, the operator is characterized
by
\begin{equation*}
  \left\{
    \begin{aligned}
      &\textup{Dom}(\mathscr{L}_{h,\textbf{A}}) = \{u \in \mathrm{H}_{h,\textbf{A}}^{1}(\RR^2) : (-i h\nabla -\A )^2 u \in \mathrm{L}^2(\RR^2)\},\\
      &\mathscr{L}_{h,\textbf{A}} u := (-i h\nabla -\A )^2 u.
    \end{aligned}
  \right.
\end{equation*}
Since $\A \in \mathscr{C}^1(\RR^2)$, the operator
$(-i h\nabla -\A )^2$ whose domain is
$\mathscr{C}_{0}^{\infty}(\RR^2)$ is essentially self-adjoint (see
\cite{FH10}).

\subsection{Context and motivation}
One of the initial motivations to study the spectral theory of the
magnetic Schr\"{o}dinger operator was the mathematical study of
superconductivity, see \cite{FH10}. The ground-energy is indeed
related to the third critical field in the Ginzburg-Landau
theory. From this initial motivation, the spectral theory of this
operator acquired a life of its own, see the series of papers by
Helffer-Morame, Helffer-Kordyukov, and Raymond-V\~u Ng\d oc
\cite{HM96,HM01,HM02,HM04,HK09,R09,HK11,HK12,HK14,RVN15,HK15,HKRVN16}. Among
this vast literature, the case of the magnetic field having a unique
and non-degenerate minimum was investigated very much. Namely, in
\cite[Theorem 1.2]{HK11} (on a compact manifold) or in \cite[Theorem
1.7]{HK15} (on $\RR^2$), Helffer and Kordyukov provided the following
asymptotic expansions
\begin{equation}\label{Eigenvalue Expansion}
  \forall \ell \in \NN, \qquad\lambda_{\ell}(\mathscr{L}_{h,\A}) = \B(p_0)h + \left(2\ell \frac{\sqrt{\det H}}{\B(p_0)} + \frac{( \mathrm{Tr}\, H^{\frac{1}{2}})^2}{2\B(p_0)} \right) h^2 + o(h^2)\,,
\end{equation}
where $p_0$ is the minimum point of $\B$ and
$H=\frac{1}{2} \mathrm{Hess} \,\B(p_0)$ .

We can also mention that, with the help of symplectic geometry and
pseudo-differential techniques, Raymond and V\~{u} Ng\d{o}c recovered
the eigenvalues expansions through a Birkhoff normal form and related
them to the magnetic classical dynamics, see \cite{RVN15}. In
\cite{BR17}, Bonthonneau and Raymond used a WKB analysis to recover
\eqref{Eigenvalue Expansion} under analyticity assumptions on the
magnetic field, when the metric is flat. Such expansions were not
known before, except in a multi-scale context, see \cite{BNHR16}. The
reader might also want to consider the well-known results about the
WKB analysis in the electric case, see \cite[Chapter 3]{DS99}.

The aim of this paper is to extend the analysis of \cite{BR17} to the
case non-flat metrics, and to remove the analyticity
assumptions. Moreover, in the case $\RR^2$ and with radial magnetic
fields, we will explicitely describe the magnetic eigenfunctions, and
establish some optimal exponential decay away from the minimum of the
magnetic field. Such an optimal decay is still a widely open problem
in the general case, see however \cite{GBRVN19, BHR20}.

\section{Statements}

\subsection{General case}\label{Manifold}
Let us state our main assumption on the magnetic field $\B$.
\begin{assumption}\label{Assumption on Manifold}
  Let $p_0\in M$, we assume that
  \begin{enumerate}[\rm (i)]
  \item The magnetic field $B\in \mathscr{C}^{\infty}(M,\RR)$ has a positive
    minimum at $p_0$, i.e.
    \[
      \B(p_0) = \min_{p \in M} \B(p) >0\,.
    \]
  \item The Hessian of $\B$ at $p_0$ is positive
    non-degenerate, i.e.
    \[
      (\dd^2 \B)_{p_0}(V,V)>0\,, \qquad \qquad \text{ for all } V \in
      T_{p_0} M \backslash \{0\} \,.
    \]
  \end{enumerate}
\end{assumption}
\begin{remark}
Note that, if $M$ is a compact manifold without boundary, our first
assumption cannot be satisfied since
\begin{equation*}
  \int_{M} \B \dd \V_{g}=\int_{M}  \dd \A = 0\,.
\end{equation*}
Nevertheless, this is not a big issue since our constructions are local.
\end{remark}
In order to state our main theorem, it is convenient to use the
isothermal coordinates, see \cite[Page 438]{Taylor11}.
\begin{definition}[Isothermal coordinates]\label{Def of Iso Coor}
  Let $(M,g)$ be a Riemannian manifold of two dimensions, a local
  chart
$$(\Omega,\phi : \Omega \to \phi(\Omega) \subset \RR^2  )$$
is called an isothermal chart if there exists a function
$\eta \in \mathscr{C}^{\infty}(\phi(\Omega) )$ such that
\begin{equation}
  \phi^{*}\left( e^{2 \eta }  g_0 \right) = g\,,
\end{equation}
where $g_0$ is the Euclidean metric on $\RR^2$.
\end{definition}

\begin{theorem}\label{IN TH WKB}
  Let $p^{*}\in M$ and assume that the magnetic field $\B$ has a local
  positive minimum at $p^{*}$ and its Hessian at $p^{*}$ is positive
  non-degenerate. Then, there exists an isothermal local chart
  $(\Omega,\phi: \Omega \to U \subset \RR^2)$ centered at $p^{*}$ in
  which the magnetic field has the form
  \[(\B\circ\phi^{-1})(q) = b_0 + \alpha q_1^2 + \gamma q_2^2 +
    \mathcal{O}(\Vert q \Vert^3)\,,\] where $b_0>0$,
  $0<\alpha\leq \gamma$ and for all $\ell\in \NN$, there exist
  \begin{enumerate}[\rm (i)]
  \item a smooth complex-valued function $P$ defined on $\Omega$
    satisfying
    \begin{equation}
      \mathrm{Re}(P\circ \phi^{-1})(q) = \frac{e^{2\eta(0)}b_0}{2}\left( \frac{\sqrt{\alpha}}{\sqrt{\alpha}+\sqrt{\gamma}}q_1^2 + \frac{\sqrt{\gamma}}{\sqrt{\alpha}+\sqrt{\gamma}}q_2^2 \right)+\mathcal{O}(\Vert q\Vert^3)\,,
    \end{equation}
    on $U$, where $\eta$ is given in Definition \ref{Def of Iso Coor},
  \item a sequence of smooth complex-valued functions
    $(U_{\ell,j})_{j\in \NN}$ defined on $\Omega$,
  \item a sequence of real numbers $(\mu_{\ell,j})_{j\in N} $ with
    \[ \mu_{\ell,0}=b_{0},\qquad \mu_{\ell,1}= \left(2\ell
        \frac{\sqrt{\det H}}{b_0} + \frac{( \mathrm{Tr}\,
          H^{\frac{1}{2}})^2}{2b_0} \right)\,,\]
  \item a sequence of smooth functions
    $(\mathit{F}_{\ell,j})_{j\in \NN}$ defined on $\Omega$ and flat at
    $p^{*}$,
  \end{enumerate}
  such that, for all $J\in \NN$,
  \[e^{P/h}\left(\mathscr{L}_{h,\A} -h\sum_{j=0}^{J}\mu_{\ell,j}h^{j}
    \right)\left(e^{-P/h}\sum_{j=0}^{J} U_{\ell,j}h^{j} \right)=
    \sum_{j=0}^{J+1} h^{j}\mathit{F}_{\ell,j}
    +\mathcal{O}(h^{J+2})\,, \] locally uniformly on $\Omega$.
\end{theorem}

\begin{corollary}\label{IN TH Exponential Smooth}
  Let $p^{*}\in M$ and assume that the magnetic field satisfies the
  conditions of Theorem \ref{IN TH WKB}. For any $\ell\in \NN$, there
  exist
  \begin{enumerate}[\rm (i)]
  \item a non-negative function
    $\widehat{P}\in \mathscr{C}_{0}^\infty(M)$ ,
  \item a sequence of functions
    $(\widehat{U}_{\ell,j})_{j\in \NN} \subset
    \mathscr{C}_{0}^\infty(M) $ ,
  \end{enumerate}
  and for any $(\varepsilon,J) \in (0,1) \times \NN $, there exist
  $C>0$ and $h_0>0$ such that, for all $h\in (0,h_0)$,
  \begin{equation}
    \Vert e^{\varepsilon \widehat{P} /h}\left(\mathscr{L}_{h,\A}-\lambda^{J}_{h,\ell} \right) \Upsilon_{h,\ell}^{J} \Vert_{\textup{L}^2(M)} \leq Ch^{J+2}\,,
  \end{equation}
  where
  \[\lambda^{J}_{h,\ell}=h\sum_{j=0}^{J}\mu_{\ell,j} h^{j} \qquad
    \text{ and }\qquad \Upsilon_{h,\ell}^{J} = \sum_{j=0}^{J}
    \widehat{U}_{\ell,j} h^{j}\,. \] In particular,
  \begin{equation}\label{IN Eigen inequality}
    \Vert \left(\mathscr{L}_{h,\A}-\lambda^{J}_{h,\ell} \right) \Upsilon^{J}_{h,\ell} \Vert_{\textup{L}^2(M)} \leq Ch^{J+2}\,.
  \end{equation}
\end{corollary}
\begin{remark}
  Corollary \ref{IN TH Exponential Smooth} can be used to prove that
  there is no odd powers of $h^{\frac{1}{2}}$ in the expansion given
  by \cite[Theorem 1.2]{HK11}. Furthermore, Corollary \ref{IN TH
    Exponential Smooth} is a generalization of \cite[Theorem
  2.1]{HK11} in the case $k=0$.
\end{remark}
Let us now turn to the description of the \enquote{true}
eigenfunctions. For each $\ell \in \NN$, let $\Upsilon_{h,\ell}$ be a
normalized eigenfunction associated with
$\lambda_{\ell}(\mathscr{L}_{h,\A})$. We introduce the projection into
the eigenspace of $\lambda_{\ell} (\mathscr{L}_{h,\A})$ :
\begin{align*}
  \Pi_{\ell} : \textup{L}^2(M) &\to \textup{Dom}(\mathscr{L}_{h,\A})\\
  u &\mapsto \Pi_{\ell} u = \left\langle u , \Upsilon_{h,\ell} \right\rangle_{\textup{L}^2(M)} \Upsilon_{h,\ell}\,.
\end{align*}
Using the asymptotic simplicity of the eigenvalues and the spectral
theorem, we get the following corollary.
\begin{corollary}\label{IN TH Eigenfunction appr}
  Assume that the magnetic field satisfies Assumption \ref{Assumption
    on Manifold}.  For all $(J,\ell)\in\NN\times \NN$, there exist
  $C>0$ and $h_0>0$ such that, for all $h\in(0,h_0)$,
  \begin{equation}
    \left\Vert \Upsilon_{h,\ell}^{J} - \Pi_{\ell} \Upsilon_{h,\ell}^{J}\right\Vert_{\textup{L}^2(M)} \leq Ch^{J+1}.
  \end{equation}
\end{corollary}

\subsection{Radial magnetic fields on $\RR^2$}\label{R2}
Let us now describe our results when $M=\RR^2$ and when the magnetic
field is radial.
\begin{assumption}\label{IN Assumption beta}
  We assume that the magnetic field $\B$ has the form
  \begin{equation*}
    \B(q_1,q_2) = \beta\left(\frac{q_1^2+q_2^2}{2}\right)\,,
  \end{equation*}
  where $\beta : \RR \rightarrow \RR^{+} $ is a smooth function such
  that
  \begin{equation}\label{QT EQ Condition Of Beta}
    \beta(r)>\beta(0)\,,  \qquad \text{ for all } r>0 \,.
  \end{equation}
  \begin{equation}\label{IN EQ Positive}
    \beta'(0)>0 \,.
  \end{equation}

\end{assumption}

In this case, the WKB analysis becomes quite explicit.
\begin{theorem}\label{IN TH WKB R}
  For all $m\in \NN$, there exist
  \begin{enumerate}[\rm (i)]
  \item a smooth positive function $\varphi$ defined on $[0,+\infty)$,
    \begin{equation}
      \varphi(\rho) := \frac{1}{2} \int_{0}^{\rho} \int_{0}^{1}  \beta(\xi \tau)\, \dd \xi \dd \tau\,,
    \end{equation}
  \item a sequence of smooth real-valued functions
    $(a_{m,j})_{j\in \NN}$ defined on $[0,\infty)$, with $a_{m,0}>0$,
  \item a sequence of real numbers $(\mu_{m,j})_{j\in N} $ with
    \[ \mu_{m,0}=\beta(0),\qquad \mu_{m,1}= \left(2m \frac{\sqrt{\det
            H}}{b_0} + \frac{( \mathrm{Tr}\, H^{\frac{1}{2}})^2}{2b_0}
      \right)\,,\quad H=\frac{1}{2} \textup{Hess} \B(0)\,.\]
  \end{enumerate}
  We let
  \begin{equation*}
    P(q) = \varphi\left( \frac{\Vert q \Vert^2}{2} \right)\,,\quad U_{m,j}(q)= a_{m,j} \left( \frac{\Vert q \Vert^2}{2} \right)\,,\quad \theta(q)=\arg(q_1+iq_2)\,.
  \end{equation*}
  Then, for all $J\in \NN$,
  \begin{eqnarray*}
    e^{P/h}\left( \frac{\Vert q\Vert^2}{2}\right)^{\frac{-m}{2}}\left(\mathscr{L}_{h,\A} -h\sum_{j=0}^{J}\mu_{m,j}h^{j} \right)\left(e^{im\theta(q)}\left( \frac{\Vert q\Vert^2}{2}\right)^{\frac{m}{2}}e^{-P/h}\sum_{j=0}^{J} U_{m,j} h^{j}\right)\\
    = \mathcal{O}(h^{J+2})\,, 
  \end{eqnarray*}
  locally uniformly in $\RR^2$.
\end{theorem}
Let $K>0$. We consider a smooth cut-off function such that
\begin{equation}\label{IN EQ Cut-off}
  \chi(\rho) = \left\{\begin{aligned}
      &1 \qquad \text{ on } [0,K]\\
      &0 \qquad \text{ on } [K+1, +\infty)
    \end{aligned} \right.\,.
\end{equation}
\begin{corollary}\label{IN TH Ex R}
  For all $(\varepsilon,m,J)\in (0,1)\times \NN \times \NN$, there
  exist a constant $C>0$ and $h_0>0$ such that, for all
  $h\in (0,h_0)$,
  \begin{equation}
    \Vert e^{\varepsilon P /h}\left(\mathscr{L}_{h,\A}-\lambda^{J}_{h,m} \right) \Upsilon_{h,m}^{J} \Vert_{\textup{L}^2(\RR^2)} \leq Ch^{J+2}\,,
  \end{equation}
  where
  \begin{align*}
    & \lambda^{J}_{h,m}:=h\sum_{j=0}^{J}\mu_{m,j} h^{j} \,,\\
    & \Upsilon_{h,m}^{J} = \chi\left( \frac{\Vert \cdot \Vert^2}{2} \right) e^{im\theta(q)}\left( \frac{\Vert q\Vert^2}{2}\right)^{\frac{m}{2}}e^{-P/h}\sum_{j=0}^{J} U_{m,j} h^{j}\,.
  \end{align*}
  In particular,
  \begin{equation}
    \Vert \left(\mathscr{L}_{h,\A}-\lambda^{J}_{h,m} \right) \Upsilon^{J}_{h,m} \Vert_{\textup{L}^2(\RR^2)} \leq Ch^{J+2}\,.
  \end{equation}
\end{corollary}

Let $\Upsilon_{h,m}$ be an eigenfunction associated with
$\lambda_{m}(\mathscr{L}_{h,\A})$. We introduce the projection into
the eigenspace of $\lambda_{m}(\mathscr{L}_{h,\A})$
\begin{align*}
  \Pi_{m} : \textup{L}^2(\RR^2) &\to \textup{Dom}(\mathscr{L}_{h,\A})\\
  u &\mapsto \Pi_{m} u = \langle u, \Upsilon_{h,m} \rangle_{\textup{L}^2(\RR^2)} \Upsilon_{h,m}\,.
\end{align*}
As in Corollary \ref{IN TH Eigenfunction appr}, we get an
approximation of the eigenfunctions.

\begin{corollary}\label{IN TH Ei Ap R}
  For all $(J,m)\in\NN\times \NN$, there exist $C>0$ and $h_0>0$ such
  that, for all $h\in(0,h_0)$,
  \begin{equation}
    \left\Vert \Upsilon_{h,m}^{J} - \Pi_{m} \Upsilon_{h,m}^{J}\right\Vert_{\textup{L}^2(\RR^2)} \leq Ch^{J+1}.
  \end{equation}
\end{corollary}
By using an Agmon estimates, we can prove that the eigenfunctions of
the magnetic Laplacian decay exponentially.
\begin{theorem}\label{IN TH L2}
  Let $U_{h,m}$ be an eigenfunction associated with
  $\lambda_{m}(\mathscr{L}_{h,\A})$. Then, for all
  $\varepsilon\in(0,1)$, there exist $C>0$ and $h_0>0$ such that, for
  all $h\in (0,h_0)$,
  \[ \Vert e^{\varepsilon P/h} U_{h,m} \Vert_{\textup{L}^2(\RR^2)}
    \leq C\Vert U_{h,m} \Vert_{\textup{L}^2(\RR^2)}\,.\]
\end{theorem}


Actually, we can even prove a stronger approximation of the
eigenfunctions.
\begin{theorem}\label{IN TH Ex Est R}
  For all $(\varepsilon,J,m)\in(0,1)\times\NN\times \NN$, there exist
  $C>0$ and $h_0>0$ such that, for all $h\in(0,h_0)$,
  \begin{equation}
    \left\Vert e^{\varepsilon P/h} \left(\Upsilon_{h,m}^{J} - \Pi_{m} \Upsilon_{h,m}^{J} \right) \right\Vert_{\textup{L}^2(\RR^2)}  \leq Ch^{J+1}\,.
  \end{equation}
\end{theorem}
Without the radial symmetry, it is expected that such a result holds
only when the magnetic field satisfies some analyticity assumption.

\subsection{Organization of the article}
The article is organized as follows. Section \ref{sec.3} is devoted to the proof of Theorem \ref{IN TH WKB}. In Section \ref{sec.4}, we prove Theorems \ref{IN TH WKB R} and \ref{IN TH Ex Est R}.

\section{Proof of Theorem \ref{IN TH WKB} and of its consequences}\label{sec.3}
\subsection{The magnetic Laplacian in isothermal coordinates.}
Let $p^{*}\in M$ be the point in Theorem \ref{IN TH WKB}. There exists
an isothermal chart $(\Omega, \phi: \Omega \to \phi(\Omega))$
centered at $p^{*}$. We set $U:=\phi(\Omega)\subset\RR^2$ and $\tilde{g} :=e^{2\eta} g_0$ (the metric on $U$). 
We let $\mathcal{M}=\phi_{*} \A$.

\begin{lemma}
  Consider the operator acting on $\textup{L}^2(U,e^{2\eta}\dd q)$
  defined by
  \[\mathcal{L}_{h,\mathcal{M}}=e^{-2\eta}\left[ \left( -ih
        \partial_{q_1} - \mathcal{M}_1 \right)^2+ \left( -ih
        \partial_{q_2} - \mathcal{M}_2 \right)^2\right]\,.\] We have
  \begin{equation}\label{QT EQ relation}
    \mathscr{L}_{h,\A} = \phi^{*} \mathcal{L}_{h,\mathcal{M}}\,.
  \end{equation}
  Moreover,
  \begin{align*}
    \left(\frac{\partial \mathcal{M}_2}{\partial q_1} -\frac{\partial \mathcal{M}_1}{\partial q_2}\right) \dd q_1 \wedge \dd q_2 = \dd (\phi_{*}\A) = \phi_{*}\dd\A = \phi_{*} \left( \B \dd\mathcal{V}_{g} \right) = e^{2\eta}( \B\circ \phi^{-1})   \dd q_1 \wedge \dd q_2 \,.
  \end{align*}
\end{lemma}

\begin{proof}
  For all $u,v\in \mathscr{C}_{0}^{\infty}(\Omega)$, we have
  \begin{align*}
    \int_{\Omega} g^*((-ih\mathrm{d}-\A)u,(-ih\mathrm{d}-\A)v) \dd \mathcal{V}_{g}  &=\int_{U} \phi_{*}\left( g^*((-ih\mathrm{d}-\A)u,(-ih\mathrm{d}-\A)v) \dd \mathcal{V}_{g}  \right)\\
                                                                                    &= \int_{U} \tilde{g}^*((-ih\mathrm{d}-\varphi^{*}\A)\tilde{u},(-ih\mathrm{d}-\varphi^{*}\A)\tilde{v} ) \vert \tilde{G} \vert^{\frac{1}{2}} \dd q\,,
  \end{align*}
  where $\tilde{u}:=\phi_{*} u$, $\tilde{v}:= \phi_{*} v$ and $\tilde{G}=\begin{pmatrix}
  e^{-2\eta} && 0\\
  0 && e^{-2\eta}
  \end{pmatrix}$ is the matrix of $\tilde{g}$.\\
  By considering $\mathcal{M}$ as a vector field
  $(\mathcal{M}_1,\mathcal{M}_2)^{T}$, we have
  \begin{align*}
    & \int_{\Omega} g^*((-ih\mathrm{d}-\A)u,(-ih\mathrm{d}-\A)v)\, \dd \mathcal{V}_{g} \\
    & = \int_{U}\left\langle \tilde{G}^{-1}(-ih \nabla_{q}-\mathcal{M})\tilde{u},(-ih \nabla_{q}-\mathcal{M})\tilde{v}  \right\rangle_{\mathbb{C}^2} \vert \tilde{G} \vert^{\frac{1}{2}} \,\dd q \\
    &= \int_{U}   \left[\left(-ih \partial_{q_1} - \mathcal{M}_1 \right)^2   + \left(-ih \partial_{q_2} - \mathcal{M}_2 \right)^2 \right]\tilde{u} \overline{\tilde{v}} \, \dd q \,,
  \end{align*}
  where the last equality is obtained by an integration by parts.
\end{proof}

\subsection{Spectral analysis with the WKB method}\label{Spectral
  Problem}

\subsubsection{A choice of the magnetic potential}\label{Subsection MP}
Let us denote $\mathcal{B}(q)=\B(\phi^{-1}(q))$. Through a linear
change of variable in $\RR^2$, without loss of generality, we can
write
\[
  \mathcal{B}(q_1,q_2)= b_0+\alpha q_1^2+\gamma q_2^2
  +\mathcal{O}(\Vert q\Vert^3) \qquad \text{ with } b_0>0 \text{ and }
  0<\alpha\leq \gamma\,.
\]
The following lemma will be useful to define a special vector
potential.
\begin{lemma}\label{Poisson Lemma}
  There exists a smooth solution of
  \begin{equation}\label{Eq Poisson}
    \Delta \Psi =e^{2\eta}\mathcal{B}\,,
  \end{equation}
  in a neighborhood of $U$ such that
  \[
    \Psi(q_1,q_2) =
    \frac{e^{2\eta(0)}\mathcal{B}(0)}{4}(q_1^2+q_2^2)+\mathcal{O}(\Vert
    q \Vert^3)\,.
  \]
\end{lemma}
\begin{proof}
  It is well-known that the Poisson equation \eqref{Eq Poisson} always
  has smooth solutions modulo harmonic functions. Consider such a
  solution $u$. We have
  \[
    u(q_1,q_2) = u(0) + \frac{\partial u(0)}{\partial{q_1}}q_1
    +\frac{\partial u(0)}{\partial{q_2}}q_2 + \frac{1}{2}
    \frac{\partial^2 u(0)}{\partial{q_1^2}}q_1^2+ \frac{\partial^2
      u(0)}{\partial{q_1q_2}}q_1 q_2 +\frac{1}{2} \frac{\partial^2
      u(0)}{\partial{q_2^2}}q_2^2 +\mathcal{O}(\Vert x\Vert^3)\,,
  \]
  and consider the harmonic function
  \[
    \varphi(q) = -\left[u(0)+ \frac{\partial u(0)}{\partial{q_1}}q_1
      +\frac{\partial u(0)}{\partial{q_2}}q_2+\frac{\partial^2
        u(0)}{\partial{q_1q_2}}q_1 q_2
    \right]+\frac{1}{4}\left(\frac{\partial^2 u(0)}{\partial{q_2^2}}-
      \frac{\partial^2 u(0)}{\partial{q_1^2}}\right)(q_1^2-q_2^2)\,.
  \]
  Then, $\Psi=u+\varphi$ satisfies
  \[
    \Psi(q) = a(q_1^2+q_2^2) +\mathcal{O}(\Vert q\Vert^3)\,,\quad 
    a=\frac{\partial^2 \Psi(0)}{\partial{q_1}^2}=\frac{\partial^2
      \Psi(0)}{\partial{q_2}^2}\,.
  \]
  Therefore, from \eqref{Eq Poisson} at $0$, we see that
  $a=\frac{e^{2\eta(0)}\mathcal{B}(0)}{4}$.
\end{proof}
Let $\Psi$ be the function given by Lemma \ref{Poisson Lemma}, we let
$A:= (-\partial_{q_2} \Psi, \partial_{q_1} \Psi)$. Then $A$ satisfies
\[
  \frac{\partial \mathcal{M}_2}{\partial q_1} -\frac{\partial
    \mathcal{M}_1}{\partial q_2} =\frac{\partial A_2}{\partial q_1}
  -\frac{\partial A_1}{\partial q_2} \,.
\]
The Poincaré lemma implies the existence of a function
$\theta\in \textup{C}^{\infty}(U)$ such that
\[
  A=\mathcal{M}+ \nabla \theta\,.
\]
By gauge invariance, we get
\begin{equation}\label{eq.gauge}
  \mathcal{L}_{h,A} = e^{i\theta/h} \mathcal{L}_{h,\mathcal{M}} e^{-i\theta/h}\,.
\end{equation}
Note that, with the choice of the magnetic potential
$A=(-\partial_{q_2} \Psi, \partial_{q_1} \Psi)$, we have
$\nabla\cdot A =0$.

\subsubsection{WKB analysis}
The eigenvalue equation reads
\begin{equation}
  \mathcal{L}_{h,A} u(q,h) = \lambda(h) u(q,h)\,.
\end{equation}
We look for a solution $u(q,h)$ in the form
\[
  u(q,h) = e^{-S(q)/h} a(q,h)\,,
\]
where $a$ and $S$ are complex-valued functions. We have
\begin{equation}\label{Eigen problem}
  \left( \mathcal{L}_{h,A}^{S}-\lambda(h)\right)\, a(q,h) = 0\,,\quad \mathcal{L}_{h,A}^{S} = e^{S/h}\mathcal{L}_{h,A}e^{-S/h}\,.
\end{equation}
We have
\begin{align*}
  \mathcal{L}_{h,A}^{S}\,  &= e^{-2\eta} e^{S/h}(-ih \nabla - A)^2e^{-S/h}\\
                           &= e^{-2\eta} \left[\left(-A_1+ i \partial_{q_1}S \right)^2 +\left(-A_2+ i \partial_{q_2}S \right)^2 +ih \nabla \cdot A  +h \Delta S+2h(\nabla S+iA) \cdot \nabla - h^2 \Delta \right]\,.
\end{align*} 
Since $\nabla\cdot A=0$, gathering the terms of same order in $h$, we
can write $\mathcal{L}_{h,A}^{S}$ as
\begin{equation}
  \mathcal{L}_{h,A}^{S} = E_0^{S} + h E_{1}^{S} -h^2 \Delta,
\end{equation}
where
\begin{equation}
  \begin{split}
    E_{0}^{S} &= e^{-2\eta} \left[\left(-A_1+ i \partial_{q_1}S \right)^2 +\left(-A_2+ i \partial_{q_2}S \right)^2\right]\,,\\
    E_{1}^{S}  &= e^{-2\eta}\left(\Delta S+2(\nabla S+iA) \cdot
      \nabla \right)\,.
  \end{split}
\end{equation}
We look for $\lambda(h)$ and $a(q,h)$ in the form
\begin{equation}\label{Eigenfunction}
  \lambda(h) = h \sum_{j=0}^{\infty} \mu_j h^{j}, \qquad a(q,h) = \sum_{j=0}^{\infty} a_j(q) h^{j}\,,
\end{equation}
where $(a_{j})_{j\geq 0}$ are smooth complex-valued functions and $(\mu_{j})_{j\in \NN}\subset \RR$.\\
Let us substitute \eqref{Eigenfunction} into \eqref{Eigen problem}, we
get the sequence of equations
\begin{align*}
  &h^0 : &&E_0^{S} a_0 = 0\\
  &h^1 : &&E_0^{S} a_1 + \left(E_1^{S} -\mu_0 \right)a_0 = 0\,,\\
  &h^2 : &&E_0^{S} a_2 + \left(E_1^{S} -\mu_0 \right)a_1 = \left(\mu_1+e^{-2\eta}\Delta  \right)a_0\,,\\
  &\qquad &&.............................................\\
  &h^n : &&E_0^{S} a_n + \left(E_1^{S} -\mu_0 \right)a_{n-1}= \left(\mu_1+e^{-2\eta}\Delta \right)a_{n-2}+\sum_{j=2}^{n-1}\mu_ja_{n-1-j}\,.
\end{align*}
In \cite{BR17}, under an analyticity assumption, a similar system of
equations was solved exactly in a neighbourhood of $0$. Here, without
an analyticity assumption, we will only be able to solve this exactly
in a space of formal series at $0$.
\subsection{WKB construction}\label{QT ST WKB}
\begin{notation}
  Let $\hat{f}$ be the Taylor formal series of
  $f \in \mathscr{C}^{\infty}(\RR^2,\mathbb{C})$ at zero, \emph{i.e.},
  \[
    \hat{f}(q_1,q_2) = \sum_{m,n\geq 0}\frac{1}{m!n!}
    \frac{\partial^{m+n} f(0)}{\partial q_1^m \partial q_2^n} q_1^m
    q_2^n.
  \]
  Let $\tilde{f}$ be the formal series after changing variable
  $(q_1,q_2) =\left(\frac{z+w}{2}, \frac{z-w}{2i} \right)$ with
  $(z,w)\in \mathbb{C}^2$. $\mathbb{C}[[z]]$ denotes the ring of
  formal series in the variable $z$ with coefficients in $\mathbb{C}$
  and $\mathbb{C}[[z,w]]$ is the ring of formal series in the variable
  $(z,w)$ with coefficients in $\mathbb{C}$.
\end{notation}
\subsubsection{The eikonal equation}
Let us find $\hat{S}$ in $\mathbb{C}[[q_1,q_2]]$ such that
\[ \left(-\hat{A}_1+ i \partial_{q_1}\hat{S} \right)^2
  +\left(-\hat{A}_2+ i \partial_{q_2}\hat{S} \right)^2 = 0 ,\] and
thus such that
\[ \left( -\hat{A}_1+ i \partial_{q_1}\hat{S} +i(-\hat{A}_2+ i
    \partial_{q_2}\hat{S})\right)\left( -\hat{A}_1+ i
    \partial_{q_1}\hat{S} -i(-\hat{A}_2+ i
    \partial_{q_2}\hat{S})\right)=0 \,.\] 
Let us consider an $\hat{S}$
such that
\[ -\hat{A}_1+ i \partial_{q_1}\hat{S} +i(-\hat{A}_2+ i
  \partial_{q_2}\hat{S}) =0\,.\] It satisfies
\[2\partial_{\bar{z}} \hat{S} = -i \hat{A}_1+\hat{A}_2, \qquad \text{
    with }\quad \partial_{\bar{z}}:= \frac{1}{2} (\partial_{q_1} +
  i\partial_{q_2} ). \] Notice that we also have
$\partial_{\bar{z}} \hat{\Psi} = -i \hat{A}_1+\hat{A}_2$. It implies
that
\begin{equation}\label{QT EQ Bar S}
  \partial_{\bar{z}} \hat{S} =\partial_{\bar{z}} \hat{\Psi}\,.
\end{equation}
After changing the variable $q_1 = \frac{z+w}{2}$ and
$q_2=\frac{z-w}{2i}$ in the formal series $\hat{S}$ and $\hat{\Psi}$,
we have $\partial_{w}=\partial_{\bar{z}}$. Thus, $\hat{S}$ satisfies
\eqref{QT EQ Bar S} if and only if $\tilde{S}$ satisfies
\[\partial_{w} \tilde{S}(z,w) = \partial_{w} \tilde{\Psi}(z,w)\,,\]
and thus $\tilde{S}$ has the form
\begin{equation}\label{QT EQ Choice S}
  \tilde{S}(z,w) = \tilde{\Psi}(z,w)+ f(z)\,,
\end{equation}
where $\displaystyle f(z)= \sum_{m \geq 0} f_{m} z^{m}$ be a formal
series in $\mathbb{C}[[z]]$ to be determined later. Next, we write the
transport equations. With the choice of $\tilde{S}$ in \eqref{QT EQ
  Choice S}, we have
\begin{equation}
  \Delta \tilde{S} = \Delta \tilde{\Psi} = \widetilde{\mathcal{E}^{-1}}\tilde{\mathcal{B}}\,\qquad \mathcal{E}= e^{-2\eta}\,.
\end{equation}
The term $(\nabla S+iA) \cdot \nabla$ in $E_{1}^S$ can be written as
\begin{align*}
  &(\partial_{q_1} \hat{S} +i\hat{A}_1) \partial_{q_{1}}+(\partial_{q_2} \hat{S} +i\hat{A}_2) \partial_{q_2} = 2\left(2\partial_{z} \tilde{\Psi} +f'(z)\right)\partial_{w}\,.
\end{align*}
The operator $E_{1}^{S}$ becomes
\begin{align*}
  E_{1}^{\tilde{S}} &= 4\tilde{\mathcal{E}}\left(2\partial_{z} \tilde{\Psi}+f'(z) \right)\partial_{w}+\tilde{\mathcal{B}}\,.
\end{align*}
Finally, we obtain the system of the transport equations:
\begin{align*}
  &h^1 : &&\left[4\tilde{\mathcal{E}}\left(2\partial_{z} \tilde{\Psi}+f'(z) \right)\partial_{w}+\tilde{\mathcal{B}} -\mu_0 \right]A^{(0)} = 0\,,\\
  &h^2 : &&\left[4\tilde{\mathcal{E}}\left(2\partial_{z} \tilde{\Psi}+f'(z) \right)\partial_{w}+\tilde{\mathcal{B}} -\mu_0 \right]A^{(1)} = \left(\mu_1+4\tilde{\mathcal{E}}\partial_{z}\partial_{w}  \right)A^{(0)}\,,\\
  &\qquad &&.............................................\\
  &h^n : &&\left[4\tilde{\mathcal{E}}\left(2\partial_{z} \tilde{\Psi}+f'(z) \right)\partial_{w}+\tilde{\mathcal{B}} -\mu_0 \right]A^{(n-1)}= \left(\mu_1+4\tilde{\mathcal{E}}\partial_{z}\partial_{w} \right)A^{(n-1)}+\sum_{j=2}^{n-1}\mu_jA^{(n-1-j)}\,.
\end{align*}
\subsubsection{Some lemmas}
In this subsection, we prove some useful lemmas for solving the
transport equations.
\begin{lemma}\label{function w}
  There exists a formal series $w(z) = \sum_{k\geq 1} w_{k} z^{k} $ in
  $\mathbb{C}[[z]]$ satisfying
  \begin{equation}\label{Eq of w}
    \tilde{\mathcal{B}}(z,w(z)) = b_0\,, 
  \end{equation}
  and such that
  $w_1 =
  \frac{\sqrt{\gamma}-\sqrt{\alpha}}{\sqrt{\gamma}+\sqrt{\alpha}}$.
\end{lemma}
\begin{proof}
  We write
  $\tilde{\mathcal{B}}(z,w) = \sum_{m,n\geq 0}
  \tilde{b}_{mn}z^{m}w^{n}$ and
  \begin{align*}
    \sum_{m,n\geq 0} \tilde{b}_{mn}z^{m} \left(\sum_{k\geq 1} w_{k}z^{k} \right)^n = b_0\,.
  \end{align*}
  Note that
  \begin{align*}
    \tilde{\mathcal{B}}(z,w) = b_0 + \frac{1}{4}(\alpha-\gamma) z^2 +\frac{1}{2}(\alpha+\gamma) zw+ \frac{1}{4}(\alpha-\gamma) w^2 +...
  \end{align*}
  Collecting the various terms, we have
  \begin{enumerate}[(a)]
  \item Term $z^{0}$ : $\tilde{b}_{00}= b_0$.
  \item Term $z^{1}$ : $\tilde{b}_{10} + \tilde{b}_{01}w_{1}= 0$.
  \item Term $z^{2}$ :
    $\tilde{b}_{02} w_1^2 + \tilde{b}_{11}w_1 + \tilde{b}_{20}=0$.
    There are two solutions for $w_1$:
    \begin{align*}
      \frac{\sqrt{\gamma}+\sqrt{\alpha}}{\sqrt{\gamma}-\sqrt{\alpha}}\,,\qquad \text{ and }\qquad \frac{\sqrt{\gamma}-\sqrt{\alpha}}{\sqrt{\gamma}+\sqrt{\alpha}}\,.
    \end{align*}
    We choose
    $w_1 =
    \frac{\sqrt{\gamma}-\sqrt{\alpha}}{\sqrt{\gamma}+\sqrt{\alpha}}$.
  \item Term $z^{3}$: notice that the equation obtained by collecting
    the coefficients of the term $z^3$ does not contain $w_k$ for
    $k\geq 3$ since $\tilde{b}_{01} =0$. This equation is linear in
    $w_2$ whose prefactor is
    \[\tilde{b}_{11}+ 2\tilde{b}_{02}w_1 =
      \frac{1}{2}(\alpha+\gamma)+\frac{1}{2}(\alpha-\gamma)\frac{\sqrt{\gamma}-\sqrt{\alpha}}{\sqrt{\gamma}+\sqrt{\alpha}}
      = \sqrt{ \alpha \gamma}> 0\,. \]
  \end{enumerate}
  By induction, let $p \in \NN\backslash \{0\}$, we assume that
  $(w_{k})_{1\leq k\leq p-1}$ are determined and we need to look for
  $w_{p}$. We collect all coefficients of $z^{p+1}$, and since
  $b_{01}=0$ , we get an equation containing only a finite number of
  terms $(w_{k})_{1\leq k \leq p}$ and
  $(\tilde{b}_{mn})_{0\leq m,n \leq p+1}$. The equation is linear in
  $w_{p}$ whose prefactor is
  \[ \tilde{b}_{11}+ 2\tilde{b}_{02}w_1=\sqrt{ \alpha \gamma} \neq
    0.\] So, $w_p$ is determined.
\end{proof}
\begin{lemma}\label{ODE}
  Let $V(s,t)$ and $F(s,t)$ be formal series in
  $\mathbb{C}[[s,t]]$. We write $V(s,t)$ and $F(s,t)$ in the form
  \[ V(s,t):=\sum_{m\geq0} v_{m}(s) t^{m}\,,\qquad F(s,t):=\sum_{m
      \geq 0} f_{m}(s) t^{m}\,,\] where $(v_n(s))_{n\in \NN}$ and
  $(f_n(s))_{n\in \NN}$ are the sequences in $\mathbb{C}[[s]]$. We
  assume that $v_0(s)=0$, $v_1(s)=v_1$ , $f_0(s)=f_0$ with
  $v_1 \in \RR\setminus \{0\}$ and $f_0\in \RR$.  Let $\ell \in \NN$,
  then
  \begin{enumerate}[\rm (i)]
  \item the homogeneous equation
    \begin{equation}\label{homogeneous eq}
      (V(s,t) \partial_t + F(s,t)) u(s,t) = 0 \,,
    \end{equation}
    has solutions in the set
    \[W(\ell)= \left\{\sum_{m\geq0} w_{m}(s) t^{m} \in
        \mathbb{C}[[s,t]]: w_{k}(s) =0 \,\text{ for } k\in
        \{0,...,(\ell-1)\} \text{ and } w_{\ell}(s)\neq 0 \right\}\]
    if and only if $f_0+\ell v_1=0$.
  \item Under this condition, there exists a family
    $(c_k(s))_{k=0...\ell} \subset \mathbb{C}[[s]] $ such that the
    inhomogeneous equation
    \begin{equation}\label{inhomogeneous eq}
      (V(s,t) \partial_z + F(s,t)) u(s,t) = G(s,t)=\sum_{m \geq 0} g_{m}(s) t^{m} \,,
    \end{equation}
    has a formal solution in the form
    \begin{equation}\label{QT EQ U}
      u(s,t)= \sum_{m \geq 0} u_{m}(s) t^{m}\,,
    \end{equation}
    if and only if
    \begin{equation}\label{inhomogenous cond}
      c_{\ell}(s) g_0(s) +c_{\ell-1}(s)g_1(s)+...+c_0(s)g_{\ell}(s) =0.
    \end{equation}
    Here, the coefficients
    $(c_k(s))_{k=0...\ell}\subset \mathbb{C}[[s]]$ are determined by
    $(v_{j}(s))_{1\leq j\leq (\ell+1)}$ and
    $(f_j(s))_{1\leq j \leq \ell}$, and $c_0(s)= 1$. Furthermore,
    assume that \eqref{inhomogenous cond} is satisfied, if
    $u_{\ell}(s)$ is given, the formal series solution $u$ will be
    determined uniquely by the relation
    \begin{equation*}
      u_{m}(s)= \frac{g_m(s)-\sum_{j=0}^{m-1} ( jv_{m-j+1}(s)+f_{m-j}(s))u_{j}(s) }{(m-\ell)v_1}\,,
    \end{equation*}
    for all $m\in\NN \setminus \{ \ell\}$.
  \end{enumerate}
\end{lemma}
\begin{proof}
  Let us start with the homogeneous equation. We look for a solution
  $u(s,t)$ in the form
  \[ \displaystyle u(s,t) = \sum_{m\geq 0} u_{m}(s) t^{m}\,, \] of the
  equation
  \[ \left(\sum_{m\geq 1} v_{m}(s) t^{m} \right) \left( \sum_{m\geq 1}
      m u_{m}(s) t^{m-1}\right)+\left(\sum_{m\geq 0} f_{m}(s) t^{m}
    \right) \left( \sum_{m\geq 0} u_{m}(s) t^{m}\right)=0\,. \] For
  arbitrary $k\in \NN$, we get the equations corresponding to $t^{k}$
  \begin{equation*}
    (kv_1+f_0)u_{k}(s)  +\sum_{j=0}^{k-1} ( jv_{k-j+1}(s)+f_{k-j}(s))u_{j}(s)  =0\,.
  \end{equation*}
  Let us write the first three equations
  \begin{align*}
    &t^0: \qquad && f_0 u_0(s) =0\,, \\ 
    &t^{1}: \qquad && \left[ v_1+f_0 \right]u_1(s) + f_1(s) u_0(s)=0\,,\\
    &t^{2}: \qquad && \left[ 2v_1+f_0 \right]u_2(s) + f_2(s) u_0(s) + \left( v_2(s)+f_1(s) \right) u_1(s) =0\,.
  \end{align*}
  For $\ell \in \NN$, consider a non-zero solution $u$ in
  $W(\ell)$. Then $u_{k}(s)=0$ for $0\leq k \leq \ell-1$ and
  $u_{\ell}(s)\neq 0$. We have $\ell v_1 + f_0 =0$. Now, if
  $\ell v_1 + f_0 =0$ , then for all $k \neq \ell$, we get
  \begin{align*}
    u_{k}(s) =-\frac{  \sum_{j=0}^{k-1} ( jv_{k-j+1}(s)+f_{k-j}(s))u_{j}(s) }{(k-\ell) v_1}\,. 
  \end{align*}
  We consider two cases:
  \begin{enumerate}[a)]
  \item $\ell=0$. It leads to $f_0=0$. Since the first equation is
    $f_0 u_0(s) =0$, we can choose any $u_{0}(s)\neq 0$, and we have
    \[ u_{k}(s) = -\frac{ \sum_{j=0}^{k-1} (
        jv_{k-j+1}(s)+f_{k-j}(s))u_{j}(s)}{kv_1},\] for all
    $k \geq 1$.
  \item $\ell \neq 0$. Then $f_0$ has to be non-zero. From the first
    equation, it leads to $u_0(s) =0$. From the recursion formula, it
    implies that $u_k(s) = 0 $ for all $k<\ell$. For $k=\ell$, we get
    the equation
    \[ (k-\ell)v_1(s) u_{\ell}(s) =0 ,\] we can choose any
    $u_{\ell}(s) \neq 0$.
  \end{enumerate}
  In any case, we always get non-trivial solutions $u$ in the set
  $W(\ell)$.

  Now we consider the inhomogeneous case:
  \[ \left(\sum_{m\geq 1} v_{m}(s) t^{m} \right) \left( \sum_{m\geq 1}
      m u_{m}(s) t^{m-1}\right)+\left(\sum_{m\geq 0} f_{m}(s) t^{m}
    \right) \left( \sum_{m\geq 0} u_{m}(s) t^{m}\right)=\sum_{m \geq
      0} g_{m}(s) t^{m}\,. \] For all $k\in \NN$, the equations
  corresponding to $t^k$ are
  \[ (kv_1+f_0)u_{k}(s) +\sum_{j=0}^{k-1} (
    jv_{k-j+1}(s)+f_{k-j}(s))u_{j}(s) =g_k(s)\,.\] Assume that there
  exists $\ell \in \NN$ such that $\ell v_1 +f_0=0$, these equations
  become
  \begin{equation}\label{QT EQ k}
    (k-\ell)v_1 u_{k}(s)  +\sum_{j=0}^{k-1} ( jv_{k-j+1}(s)+f_{k-j}(s))u_{j}(s)  =g_k(s)\,.
  \end{equation}
  The equation corresponding to $t^\ell$ is
  \begin{equation*}
    (\ell -\ell)u_{\ell}(s) + \sum_{j=0}^{\ell-1} ( jv_{\ell-j+1}(s)+f_{\ell-j}(s))u_{j}(s) = g_{\ell}(s)\,.
  \end{equation*}
  The inhomogeneous equation has solutions in the form \eqref{QT EQ U}
  if and only if
  \[ g_{\ell}(s)-\sum_{j=0}^{\ell-1} (
    jv_{\ell-j+1}(s)+f_{\ell-j}(s))u_{j}(s)=0.\] This relation is in
  the form \eqref{inhomogenous cond} after computing $(u_{k}(s))$
  as a function of $g_k(s)$, $f_k(s)$ and $v_k(s)$ for
  $k =0,\ldots,\ell-1$. For example, we can compute $c_1(s)$ by
  collecting all coefficients connecting with $g_{\ell-1}(s)$. Notice
  that $g_{\ell-1}(s)$ only appears in the formula of $u_{\ell-1}$ and
  its coefficient in $u_{\ell-1}$ is $\frac{-1}{v_1}$, then we can
  compute
  \[ c_1(s) = - \left[ (\ell-1)v_2(s) + f_1(s)\right]\frac{-1}{v_1} =
    \frac{(\ell-1)v_2(s) + f_1(s)}{v_1}\,. \] The statement at the end
  of the lemma is obtained from \eqref{QT EQ k}.
\end{proof}

\subsubsection{The first transport equation}
We consider the first transport equation
\begin{equation}\label{Transport equation 1}
  v(z,w)\partial_{w} A^{(0)}(z,w) + \left(\tilde{\mathcal{B}}(z,w)-\mu_0\right)A^{(0)}(z,w)=0\,,
\end{equation}
where $v(z,w)= 4\tilde{\mathcal{E}}(z,w)\left(f'(z)+2\partial_{z} \tilde{\Psi}(z,w) \right)\,.$\\
Let $w(z)$ be the formal series given by Lemma \ref{function w}. Using
the change of variables $(z,w)=(z,y+w(z))$, which is allowed in
$\mathbb{C}[[(z,w)]]$ since $w_0=0$, we get the equation
\begin{equation}\label{QT Transport equation 1}
  V(z,y) \partial_{ y} A^{(0)}(z, y+w(z))+ F(z,y) A^{(0)}(z, y+w(z)) = 0\,,
\end{equation}
where
\[ V(z,y) = v(z,y+w(z))=
  4\tilde{\mathcal{E}}(z,y+w(z))\left(f'(z)+2\partial_{z}
    \tilde{\Psi}(z,y+w(z)) \right)\,,\] and
\[ F(z,y) = \tilde{\mathcal{B}}(z,y+w(z)) -\mu_{0}\,.\]
\subsubsection{Choosing formal series \texorpdfstring{$f$}{} and
  determining $\tilde{S}$}
We recall that $\tilde{S}$ is given in $\eqref{QT EQ Choice S}$. The
formal series $\tilde{\Psi}(z,w)$ is given by Lemma \ref{Poisson
  Lemma} and the formal series $f(z)\in \mathbb{C}[[z]]$ has to be
determined. We choose $f(z)$ such that we can apply Lemma
\ref{ODE}, \emph{i.e.},
\begin{equation}\label{QT EQ f}
  f'(z)+2\partial_{z}\tilde{\Psi}(z,w(z))=0 \,.
\end{equation}
According to Lemma \ref{Poisson Lemma},
\[ \tilde{\Psi}(z,w) = \frac{e^{2\eta(0)}b_0}{4}zw+ \sum_{m+n\geq 3}
  \psi_{mn}z^m w^n\,,\] so that
\[\partial_{z}\tilde{\Psi}(z,w(z))= \frac{e^{2\eta(0)}b_0}{4}w(z) +
  \sum_{\substack{m+n\geq 3\\ m\geq 1}} m\psi_{mn}z^{m-1} (w(z))^{n}
  \,.\] Let $ f(z) = \sum_{k \geq 0} \hat{f}_{k} z^{k} $. Since there
is no restriction for $f(0)$, we can choose $\hat{f}_{0}=0$. Moreover,
a straightforward computation using Lemma \ref{function w} gives
\[ \hat{f}_1=0 \,, \qquad
  \hat{f}_{2}=\frac{e^{2\eta(0)}b_0}{4}\frac{\sqrt{\alpha}-\sqrt{\gamma}}{\sqrt{\alpha}+\sqrt{\gamma}}\,.\]
Now, $\tilde{S}(z,w)$ is completely determined and
\begin{equation}\label{Eq SM fun S}
  \tilde{S}(z,w) = \frac{e^{2\eta(0)}b_0}{4} zw +\frac{e^{2\eta(0)}b_0}{4}\frac{\sqrt{\alpha}-\sqrt{\gamma}}{\sqrt{\alpha}+\sqrt{\gamma}} z^2 + \sum_{m+n\geq 3} [\tilde{S}]_{mn}z^m w^n\,.
\end{equation}
\subsubsection{Solving the first transport equation}
Let us come back to the transport equation \eqref{Transport equation
  1}. We write
$$\displaystyle  V(z,y) = \sum_{m\geq0}v_{m}(z) {y}^{m}\text{ and } \displaystyle F(z,y) = \sum_{m\geq0} f_{m}(z) y^{m}\,.$$
We now check the assumptions of Lemma \ref{ODE}. From \eqref{QT EQ f},
we have
\begin{equation*}
  v_{0}(z) = V(z,0)= 4\tilde{\mathcal{E}}(z,w(z))\left(f'(z)+2 \partial_{z}\tilde{\Psi}(z,w(z)) \right) =0\,.
\end{equation*}
From Lemma \ref{Poisson Lemma}, we obtain
\begin{eqnarray*}
  v_{1}(z) &=& \partial_{y} V(z,0)\\
           &=& 4 \partial_{w} \tilde{\mathcal{E}}(z,y+w(z)) \left(f'(z)+2\partial_{z} \tilde{\Psi}(z,y+w(z)) \right)\bigg|_{y=0} \\
           &\,& + 8\tilde{\mathcal{E}}(z,y+w(z))\partial_{w} \partial_z  \tilde{\Psi}(z,y+w(z))\bigg|_{y=0} \\
           &=& 2 \tilde{\mathcal{B}}(z,w(z))=2b_0\,.
\end{eqnarray*}
Finally, from \eqref{function w}, we get
\begin{equation*}
  f_{0}(z) = F(z,0)= \tilde{\mathcal{B}}(z,w(z))-\mu_0 = b_0 - \mu_0\,.
\end{equation*}
Thanks to Lemma \ref{ODE} for the homogeneous case, we see that \eqref{QT
  Transport equation 1} has solutions in the form
$\sum_{m\geq 0} A^{(0)}_{m}(z) y^{m}$ with $A^{(0)}_{0}(z)\neq 0$
if and only if $f_{0}(z) =0$, i.e. $\mu_0 = b_0$. In this case, the
solution of the first transport equation \eqref{Transport equation 1}
is in the form
$A^{(0)}(z,w)= \sum_{m\geq 0} A^{(0)}_{m}(z) (w-w(z))^{m}$, where
$ A^{(0)}_{m}(z)$ can be computed for $m\geq 1$ thanks to
\[ A^{(0)}_{m}(z) = -\frac{ \sum_{j=0}^{m-1} (
    jv_{m-j+1}(z)+f_{m-j}(z)) A^{(0)}_{j}(z) }{2mb_0}\,.\] The series
$A_{0}^{(0)}(z)$ will be determined later.
\subsubsection{The second transport equation}
We consider the second transport equation
\begin{equation}\label{Transport equation 2}
  \left(v(z,w)\partial_{w}+ \tilde{B}(z,w)-\mu_0 \right) A^{(1)} = (\mu_1 + 4\tilde{\mathcal{E}}(z,w)\partial_{z}\partial_{w}) A^{(0)}\,.
\end{equation}
By changing variables $(z,w)=(z,y+w(z))$, we obtain
\begin{align}\label{QT Transport equation 2}
  V(z,y) \partial_{ y} A^{(1)}(z, y+w(z))+ F(z,y) A^{(1)}(z, y+w(z)) = G^{(1)}(z,y) \,,
\end{align}
where $G^{(1)}(z,y)=\left(\mu_1+4 \tilde{\mathcal{E}}(z,y+w(z))\partial_{z}\partial_{w} \right)A^{(0)}(z,y+w(z))$.\\
We write $G^{(1)}(z,y)=\sum_{m\geq 0} g_{m}^{(1)}(z)y^{m}$. Since
$v_0(z)=0$, $v_1(z)=2b_0 \neq 0$ and $f_0(z)=0$, and thanks to Lemma
\ref{ODE} (with $\ell=0$), the equation \eqref{QT Transport equation
  2} has solutions if and only if $g_0^{(1)}(z)=0$, i.e.,
\[
  G^{(1)}(z,0)=(\mu_1+4\tilde{\mathcal{E}}(z,w(z))\partial_{z}\partial_{w})
  A^{(0)}(z,w(z))=0 \,,\] or, equivalently,
\[ \mu_1 A^{(0)}_{0}(z) +4\tilde{\mathcal{E}}(z,w(z))
  \left(\partial_{z}A^{(0)}_{1}(z) - 2A^{(0)}_{2}(z)w'(z)\right)=0.\]
Since
\[ A^{(0)}_{1}(z)= - \frac{f_{1}(z)}{2b_0}A^{(0)}_{0}(z)\,,\] and
\begin{align*}
  A^{(0)}_{2}(z)&=- \frac{(v_{2}(z)+f_{1}(z))A^{(0)}_{1}(z)+f_{2}(z)A_{0}^{(0)}(z)}{4b_0}\\
                &= \frac{1}{8b_0^2} f_{1}(z)(v_{2}(z)+f_{1}(z))A_{0}^{(0)}(z)-\frac{1}{4b_0}f_{2}(z)A_{0}^{(0)}(z)\,,
\end{align*}
the equation related to $\mu_1$ can be rewritten as
\begin{equation}\label{Eq of a00}
  V(z)\partial_{z} A^{(0)}_{0}(z)+F(z)A^{(0)}_{0}(z)=0\,, \qquad V(z):=\frac{2}{b_0}\tilde{\mathcal{E}}(z,w(z))f_{1}(z)\,,
\end{equation}
with
\[ F(z):=\tilde{\mathcal{E}}(z,w(z)) \left(
    \frac{2}{b_0}f_1'(z)+\frac{1}{b_0^2}f_1(z)\left(f_1(z)+v_2(z)\right)w'(z)+\frac{2}{b_0}f_2(z)w'(z)\right)-\mu_1\,.\]
\begin{notation}
  Below, with a given formal series
  $X(z)=\sum_{k\geq 0} x_k z^{k} \in C[[z]]$, we use the notation
  $[X(z)]_{k}$ to extract the coefficient of $z^{k}$, so that
  $[X(z)]_{k}=x_{k}$.
\end{notation}
We have
\begin{align*}
  f_1(z)&=\partial_2 \tilde{\mathcal{B}}(z,w(z))= \sum_{\substack{m \geq 0\\ n\geq 1}} n \tilde{b}_{mn}z^{m} \left( w(z)\right)^{n-1}\,,\\
  f_2(z)&=\frac{1}{2}\partial_2^2 \tilde{\mathcal{B}}(z,w(z)) =\sum_{\substack{m \geq 0\\ n\geq 2}} \frac{n(n-1)}{2} \tilde{b}_{mn}z^{m} \left( w(z)\right)^{n-2}\,.
\end{align*}
It is easy to check that
\begin{align*} [V(z)]_{0}=\frac{2}{b_0}
  [\tilde{\mathcal{E}}(z,w(z))]_{0}\left[f_{1}(z)\right]_{0}=0 \qquad
  (\text{since }\tilde{b}_{01} =0)\,,
\end{align*}
and
\begin{align*}
  [V(z)]_{1}&=\frac{2}{b_0} [\tilde{\mathcal{E}}(z,w(z))]_{0}\left[f_{1}(z)\right]_{1}+ \frac{2}{b_0} [\tilde{\mathcal{E}}(z,w(z))]_{1}\left[f_{1}(z)\right]_{0}\\
            &=\frac{2e^{-2\eta(0)}}{b_0} \left(2\tilde{b}_{02}w_1+\tilde{b}_{11} \right)\\
            &= \frac{2e^{-2\eta(0)}}{b_0} \left( \frac{1}{2}\left(\alpha-\gamma \right)\frac{\sqrt{\gamma}-\sqrt{\alpha}}{\sqrt{\gamma}+\sqrt{\alpha}}+\frac{1}{2}(\alpha+\gamma) \right)\\
            &= \frac{2e^{-2\eta(0)}\sqrt{\alpha\gamma}}{b_0} \,.
\end{align*}
Furthermore, we can compute
\begin{equation*} [F(z)]_0 =
  \frac{2e^{-2\eta(0)}}{b_0}\left(\tilde{b}_{11}+
    \tilde{b}_{02}w_1\right)-\mu_1=
  \frac{e^{-2\eta(0)}(\sqrt{\gamma}+\sqrt{\alpha})^2}{2b_0}-\mu_1\,.
\end{equation*}
Applying Lemma \ref{ODE}, we see that \eqref{Eq of a00} has solutions
if and only if there exists $\ell \in \NN $ such that
\begin{align*}
  \mu_1 =e^{-2\eta(0)}\left(2\ell \frac{\sqrt{\alpha\gamma}}{b_0}+ \frac{(\sqrt{\gamma}+\sqrt{\alpha})^2}{2b_0}\right).
\end{align*}
Then, $A_0^{(0)}$ can be determined as a formal series whose $\ell$
first terms vanish and $[ A_0^{(0)}(z)]_{\ell} = 1$.

Once the right-hand side of \eqref{Transport equation 2} is
determined, we can find a particular solution in the form
\[ \sum_{m\geq 0} \alpha^{(1)}_m(z) (w-w(z))^{m}\,,\] where the
$\alpha^{(1)}_m(z)$ are determined by a recursion formula starting
with $\alpha^{(1)}_0(z)=0$. The solution of \eqref{Transport equation
  2} takes the form
\begin{equation}
  A^{(1)}(z,w) = \sum_{m\geq 0} \alpha^{(1)}_m(z) (w-w(z))^{m}+ \sum_{m\geq 0} A^{(1)}_{m}(z) (w-w(z))^{m}\,,
\end{equation}
where $ \sum_{m\geq 0} A^{(1)}_{m}(z) (w-w(z))^{m}$ is the solution of
the first transport equation \eqref{Transport equation 1}, and in
which only $ A^{(1)}_{0}(z)$ remains to be determined.
\subsubsection{Induction}
Let $p\in \NN\setminus \{0\}$. We assume that the sequences
$(\mu_j)_{0\leq j\leq p}$ and $(A^{(j)})_{0\leq j\leq p-1}$ are
determined from the first $(p+1)$ transport equations:
\begin{align*}
  &\left(v(z,w)\partial_{w}+ \tilde{B}(z,w)-\mu_0 \right) A^{(0)}=0\\
  &\left(v(z,w)\partial_{w}+ \tilde{B}(z,w)-\mu_0 \right) A^{(1)} = (\mu_1 + 4\tilde{\mathcal{E}}(z,w)\partial_{z}\partial_{w}) A^{(0)}\\
  &\quad .........................................................\\
  &\left(v(z,w)\partial_{w}+ \tilde{B}(z,w)-\mu_0 \right) A^{(p)}=(\mu_1 + 4\tilde{\mathcal{E}}(z,w)\partial_{z}\partial_{w}) A^{(p-1)}+ \sum_{j=2}^{p} \mu_{j} A^{(p-j)}\,.
\end{align*}
Let us also assume that the $A^{(j)}$'s, for $j\in\{1,...,p\}$, are in
the form
\[A^{(j)}(z,w) = \sum_{m\geq 0} \alpha^{(j)}_m(z) (w-w(z))^{m}+
  \sum_{m\geq 0} A^{(j)}_{m}(z) (w-w(z))^{m}\,, \] where
\begin{enumerate}[\rm (i)]
\item $ \sum_{m\geq 0} \alpha^{(j)}_m(z) (w-w(z))^{m}$ is a particular
  solution of the $j$-th transport equation and satisfies
  $\alpha^{(j)}_0(z)=0 $ in $\mathbb{C}[[z]]$ for $j\in \{1,...,p\}$.
\item $ \sum_{m\geq 0} A^{(j)}_{m}(z) (w-w(z))^{m}$, which is a
  solution of the first transport equation \eqref{Transport equation
    1}, is also determined and satisfies $[A^{(j)}_{0}(z)]_{\ell}=0$
  in $\mathbb{C}$ for $j\in \{1,...,p-1\}$ .
\end{enumerate}
At the rank $p$, only $A^{(p)}_{0}(z) $ needs to be determined.

Let us now consider the equation satisfied by $A^{(p+1)}$:
\begin{equation}\label{Eq of A_p+1}
  \left( v\, \partial_{w}  + \tilde{B}-b_0 \right)A^{(p+1)}= (\mu_1+4\tilde{\mathcal{E}}\partial_{z}\partial_{w} )A^{(p)}+ \mu_{p+1} A^{(0)} + \sum_{j=2}^{p} \mu_{j} A^{(p+1-j)}.
\end{equation}
As before, the fact that this equation has solutions determines the
value of $\mu_{p+1}$ and the series $A^{(p)}_{0}(z)$. We leave these
details to the reader.

\subsubsection{Conclusion}
From the above analysis, we get the following theorem.
\begin{theorem}[WKB construction for
  $\mathcal{L}_{h,\mathcal{M}}$]\label{QT TH WKB}
  For all $\ell\in \NN$, there exist
  \begin{enumerate}[\rm (i)]
  \item a smooth complex-valued function $T$ on $U$ satisfying
    \begin{equation}\label{Eq SM Re S}
      \mathrm{Re}(T)(q) = \frac{e^{2\eta(0)}b_0}{2}\left( \frac{\sqrt{\alpha}}{\sqrt{\alpha}+\sqrt{\gamma}}q_1^2 + \frac{\sqrt{\gamma}}{\sqrt{\alpha}+\sqrt{\gamma}}q_2^2 \right)+\mathcal{O}(\Vert q\Vert^3)\,,
    \end{equation}
  \item a sequence of smooth complex-valued function
    $(a_{\ell,j})_{j\in \NN}$ on $U$,
  \item a sequence of real numbers $(\mu_{\ell,j})_{j\in N} $
    with
    \[ \mu_{\ell,0}=b_{0},\qquad \mu_{\ell,1}=
      e^{-2\eta(0)}\left(2\ell \frac{\sqrt{\alpha\gamma}}{b_0}+
        \frac{(\sqrt{\gamma}+\sqrt{\alpha})^2}{2b_0}\right)\,,\]
  \item a sequence of a flat function
    $(\mathit{f}_{j})_{j\in \NN}$ on $U$\,,
  \end{enumerate}
  such that, for all $J\in \NN$,
  \[e^{T/h}\left(\mathcal{L}_{h,\mathcal{M}}
      -h\sum_{j=0}^{J}\mu_{\ell,j}h^{j}
    \right)\left(e^{-T/h}\sum_{j=0}^{J} a_{\ell,j}h^{j} \right)=
    \sum_{j=0}^{J+1} h^{j}\mathit{f}_{j} +\mathcal{O}(h^{J+2})\,, \]
  local uniformly in $U$.
\end{theorem}
\begin{proof}
  We have constructed  the formal series $\tilde{S}(z,w),\, (A^{(j)}(z,w))_{j\in \NN}$ in
  $\mathbb{C}[[z,w]]$ and $\mu_{\ell,j}$, all depending on $\ell$. By
  coming back to the original variables $(q_1,q_2)$, we obtain formal
  series in $\mathbb{C}[[q_1,q_2]]$. Applying Borel's Lemma, we get a smooth complex-valued functions $S$
  and $(a_{j})_{j\in \NN}$ so that for each $J\in \NN$, there exist
  flat functions $\mathit{f}_{0},\mathit{f}_1,...,\mathit{f}_{J+1}$
  and a smooth function $F$ on $\RR^2$ ($F$ is a polynomial in $h$
  whose coefficients are smooth functions depending on $a_j$ and
  $\mu_{\ell,j}$) such that
  \[ \left(\mathcal{L}_{h,A}^{S} -h\sum_{j=0}^{J}\mu_{\ell,j}h^{j}
    \right)\left(\sum_{j=0}^{J} a_{\ell,j} h^{j}\right)=
    \sum_{j=0}^{J+1} h^{j}\mathit{f}_{j} +h^{J+2}F\,.\] Note that
  $\mathcal{L}_{h,A}^{S}=e^{S/h}\mathcal{L}_{h,A} e^{-S/h} $ so that
  \begin{equation}
    e^{S/h}\left(\mathcal{L}_{h,A} -h\sum_{j=0}^{J}\mu_{\ell,j}h^{j} \right)\left(e^{-S/h}\sum_{j=0}^{J} a_{\ell,j}h^{j} \right)= \sum_{j=0}^{J+1} h^{j}\mathit{f}_{j} +h^{J+2} F\,.
  \end{equation}
  We recall that we used the WKB method for the operator
  $\mathcal{L}_{h,A}$, with the special magnetic potential $A$, see
  Section \ref{Subsection MP} and especially \eqref{eq.gauge}. Letting
  $T=S+i\theta$, we get
  \begin{equation}\label{QT EQ WKB F}
    e^{T/h}\left(\mathcal{L}_{h,\mathcal{M}} -h\sum_{j=0}^{J}\mu_{\ell,j}h^{j} \right)\left(e^{-T/h}\sum_{j=0}^{J} a_{\ell,j}h^{j} \right)= \sum_{j=0}^{J+1} h^{j}\mathit{f}_{j}+h^{J+2}F\,.
  \end{equation}
  Since $\mathrm{Re}(T)= \mathrm{Re}(S)$, the Taylor expansion of
  $\mathrm{Re}(T)$ directly comes from \eqref{Eq SM fun S}.
\end{proof}
\begin{proof}[Proof of Theorem \ref{IN TH WKB}]
  Let us recall \eqref{QT EQ relation}. Consider a fixed cut-off
  function $\chi_1$ equal to $1$ in a neighbourhood of $0$ and with
  support in $U$. Then, we set
  \begin{enumerate}
  \item[i)] $P = \phi^{*} (\chi_1 T)$,
  \item[ii)] $U_{\ell,j} = \phi^{*} (\chi_1 a_{\ell,j})$ for all
    $(\ell,j)\in \NN^2$,
  \end{enumerate}
  Using Appendix \ref{app.B}, we deduce Theorem \ref{IN TH WKB}.

\end{proof}

In order to prove Corollary \ref{IN TH Exponential Smooth}, we just
need to shrink the supports of functions in Theorem \ref{IN TH WKB} by
inserting appropriate cutoff functions and use the fact that there
exist $R>0$ and $\delta>0$ such that
\begin{equation}\label{Eq SM fun S estimate}
  \mathrm{Re}(T)(q)\geq \delta \Vert q \Vert^2  \text{ for all } q\in D(0;R)\,.
\end{equation}
This coercivity of the phase implies that the Ansatz has an
exponential decay.

\section{Proof of Theorem \ref{IN TH WKB R} and of its consequences}\label{sec.4}

\subsection{Radial magnetic Laplacian and Fourier decomposition}
Consider the polar coordinates
\begin{equation}\label{QT EQ Radial Coor}
  \psi :\left\{ 
    \begin{aligned}
      \RR^{+}\times \RR/2 \pi\mathbb{Z} &\to \RR^2\setminus \{0\}\\
      (r,\theta)&\mapsto (r\cos \theta, r\sin \theta ) = (q_1,q_2)
    \end{aligned}\right.\,.
\end{equation}
The following proposition is easy to get.
\begin{proposition}
  Letting $\A=A_1 \dd q_1 + A_2 \dd q_2$, we have
  \[\mathbf{A}=A_{r}\dd r+A_{\theta}\dd \theta\,,\]
  with $\displaystyle \tilde{\textup{\A}}=\left(A_{r}, A_{\theta}\right)^\mathrm{T}:=(\dd\psi)^\mathrm{T}\left(A_{1}, A_{2}\right)^\mathrm{T}$, where $\dd \psi$ denotes the Jacobian matrix of $\psi$.\\
  The magnetic Laplacian $\mathscr{L}_{h,\A}$ is unitary equivalent to
  the operator
  \begin{equation}\label{QT EQ ML Radial}
    \mathscr{K}_{h,\tilde{\textup{\A}}} = r^{-2} \left(r(-ih \partial_r -A_r)\right)^2  \  + r^{-2}(-ih \partial_\theta -A_\theta)^2 \,,
  \end{equation}
  whose domain is
  \[\textup{Dom}( \mathscr{K}_{h,\tilde{\textup{\A}}}) = \left\{ w \in
      \textup{L}^2(\RR^{+}\times \RR/2 \pi\mathbb{Z} ,r \dd r \dd
      \theta) : \mathscr{K}_{h,\tilde{\textup{\A}}}w \in
      \textup{L}^2(\RR^{+}\times \RR/2 \pi\mathbb{Z} ,r \dd r
      \dd\theta) \right\}\,.\]
\end{proposition}
Thanks to the gauge of invariance, we can choose a magnetic potential
compatible with the radial symmetry (the one given by the Poincaré
lemma):
\begin{equation}\label{QT EQ Magnetic Potential}
  A_1(q)=-q_2 \alpha(q) ,\qquad A_2(q) = q_1\alpha(q) \,, 
\end{equation}
where
\[\alpha(q):= \int_{0}^{1} t \textbf{B}(tq) \,\dd t\,.\]
With the choice, we get
\[ A_r(r,\theta) = 0, \,\qquad A_{\theta}(r,\theta) = G(r):=\int_0^{r}
  \tau \beta\left(\frac{\tau^2}{2}\right) \dd \tau\,, \] and the
magnetic Laplacian becomes
\begin{equation*}
  \mathscr{K}_{h} = -h^2 r^{-2} \left(r \partial_r \right)^2  \  + r^{-2}(-ih \partial_\theta - G(r))^2\,.
\end{equation*}
Via the Fourier decomposition, the magnetic Laplacian
$\mathscr{K}_{h}$ can be written as the direct sum of radial electric
Schr\"{o}dinger operators
\begin{equation}
  \mathscr{K}_{h} =  \bigoplus_{m \in \ZZ} \mathfrak{L}_{h,m}\,,\qquad \mathfrak{L}_{h,m}:= -h^2 r^{-2} \left(r \partial_r \right)^2  \  + r^{-2}(hm- G(r))^2\,.
\end{equation}
Therefore, we can focus on the spectral analysis of
$\mathfrak{L}_{h,m}$, especially on its ground-energy.

\subsection{Compact resolvent}

\begin{proposition}\label{QT Compact resolvent electric}
  For any $m \in \ZZ$ and $h>0$, the operator $\mathfrak{L}_{h,m}$ has
  compact resolvent.
\end{proposition}
\begin{proof}
  We fix $m \in \ZZ$ and $h>0$, consider the sesquilinear
  \begin{equation*}
    \mathcal{Q}_{h,m} (u,v) = 2h^2 \int_0^{+\infty} \partial_{r} \, u \,\overline{\partial_{r}  v }\, r \dd r +  \int_0^{+\infty} \frac{\left( hm-G(r)\right)^2}{r^2} u\, \overline{v}\, r\dd r +\int_0^{+\infty} u \overline{v} r\dd r\,,
  \end{equation*}
  defined on the domain
  \begin{equation*}
    \textup{Dom}(\mathcal{Q}_{h,m}) = \left\{ u \in \textup{L}^2(\RR^{+},r\dd r) : \partial_{r}u \in \textup{L}^2(\RR^{+},r\dd r) , \sqrt{V_{h,m}}u\in \textup{L}^2(\RR^{+},r\dd r)\right\},
  \end{equation*}
  with
  \[ V_{h,m}(r) := \frac{\left( hm-G(r)\right)^2}{r^2}.\]
  $ \textup{Dom}\left(\mathcal{Q}_{h,m} \right) $ is a Hilbert space
  equipped with the inner product $\mathcal{Q}_{h,m}$. The
  sesquilinear form $\mathcal{Q}_{h,m}$ induces a positive
  self-adjoint operator $\mathcal{S}_{h,m}$. From \eqref{QT EQ
    Condition Of Beta}, we have
  \begin{equation}\label{QT EQ V To Infty}
    \lim_{r\to +\infty} \frac{(hm-G(r))^2}{r^2} = +\infty\,.
  \end{equation}
  The conclusion easily follows by means of the
  Riesz-Fréchet-Kolmogorov criterion for compactness in $L^p$ spaces.
\end{proof}
\subsection{Spectrum of rescaled radial electric Schr\"{o}dinger
  operators}\label{Section rescale op}
By the change of variable $\rho=\frac{r^2}{2}$, we have
\[ G(r)=\int_{0}^{r} \tau \beta\left(\frac{\tau^2}{2} \right)\,\dd
  \tau = \int_{0}^{\rho} \beta(s)\,\dd s\,.\] We let
\[ a(\rho):= \int_{0}^{\rho} \beta(s)\,\dd s \,.\] Using the unitary
transformation
\begin{align}\label{QT EQ Unitary T1}
  T_1 : & \textup{L}^2(\RR^{+},d\rho)\to \textup{L}^2(\RR^{+},r\dd r), \\
        & v(\rho) \mapsto (T_1 v) (r) = v(r^2/2)\,,\nonumber
\end{align}
we get the new operator, acting on $\textup{L}^2(\RR^{+},\dd\rho)$,
\begin{equation}
  \mathcal{N}_{h,m}:= T_1^{-1}\mathfrak{L}_{h,m}T_1= -2h^2 \partial_{\rho} \rho \partial_{\rho} + \frac{\left( hm-a(\rho)\right)^2}{2\rho}\,.
\end{equation}

\subsubsection{Rescaling}
The rescaling $\rho = ht$ is convenient to analyse the expansion of
the eigenvalues and some decay properties of the
eigenfunctions. Consider
\begin{align}\label{QT EQ Unitary T2}
  T_2 : & \textup{L}^2(\RR^{+},\dd t)\to \textup{L}^2(\RR^{+},\dd \rho), \\
        & v(t) \mapsto (T_2 v) (\rho) = h^{-1/2} v(h^{-1}\rho).\nonumber
\end{align}
and
\begin{equation}
  \mathcal{M}_{h,m}:= T_2^{-1} \mathcal{N}_{h,m} T_2 = -2h \partial_{t} t\partial_{t} + \frac{(hm-a(ht))^2}{2ht}\,,
\end{equation}
We have
\begin{eqnarray*}
  a(ht) = \beta(0) ht + \frac{\beta'(0)}{2} h^2t^2 + \mathcal{O}(h^{3}t^{3})\,.
\end{eqnarray*}
and thus
\begin{equation*}
  \mathcal{M}_{h,m} = h \mathfrak{L}_{m}^{[0]} + h^2 \mathfrak{L}_{m}^{[1]} +\frac{R(ht)}{t},
\end{equation*}
where
\begin{align*}
  \mathfrak{L}_{m}^{[0]} &:= -2 \partial_{t} t\partial_{t} + \frac{\beta(0)^2t}{2}  +\frac{m^2}{2t} - m \beta(0)\\
  \mathfrak{L}_{m}^{[1]} &:= -\frac{m\beta'(0)}{2} t +\frac{\beta(0)\beta'(0)}{2} t^2 \,,
\end{align*}
and $R$ is a remainder such that there exist a constant $C>0$
and $\delta>0$ such that
\begin{equation}\label{delta radial}
  \vert R(s) \vert \leq C s^3 \qquad \text{ for all } s \in [0,\delta)\,.
\end{equation}
We consider $\mathfrak{L}_{m}^{[0]}$ defined through the sesquilinear
form
\begin{equation*}
  \mathcal{Q}_{m}^{[0]}(u,v)= \int_{0}^{+\infty} 2 t\partial_t u\, \overline{\partial_t v} \dd t + \int_{0}^{+\infty} \left( \frac{\beta(0)^2t}{2}  +\frac{m^2}{2t} - m \beta(0) \right) u\, \overline{v} \dd t\,.
\end{equation*}
The operator $\mathfrak{L}_{m}^{[0]}$ has clearly compact
resolvent. Furthemore, the spectrum of $\mathfrak{L}_{m}^{[0]}$ is
well known and related to the Laguerre operator. Indeed, by letting
$t =\frac{s}{\beta(0)}$ , the operator $ \mathfrak{L}_{m}^{[0]}$
becomes
\begin{equation}
  \beta(0) \left( -2\partial_{s} s\partial_{s} + \frac{s}{2}  +\frac{m^2}{2s} -m \right).
\end{equation}
We consider the following operator $\mathcal{T}_{m}$ on the Hilbert
space $\textup{L}^2(\RR^{+},s^{-\vert m \vert} e^{s} ds) $
\begin{equation*}
  \mathcal{T}_{m} = s^{-\frac{\vert m \vert}{2}}e^{\frac{s}{2}}\left( -2\partial_{s} s\partial_{s} + \frac{s}{2}  +\frac{m^2}{2s} - m \right)s^{\frac{\vert m \vert}{2} }e^{-\frac{s}{2}}\,.
\end{equation*}
It is unitarily equivalent to $\mathfrak{L}_{m}^{[0]}$. A computation
gives
\begin{eqnarray*}
  \mathcal{T}_{m}= -2s \partial_{s}^2 +\left(2s-2-2\vert m \vert \right)\partial_{s} +\vert m \vert -m+1 \,.
\end{eqnarray*}
The spectrum of operator $\mathcal{T}_{m}$ is described in Appendix
\ref{app.B} and we get
\[ \textup{Sp}(\mathfrak{L}_{m}^{[0]})=\left\{(2k+1+\vert m \vert-m)
    \beta(0) : k\in \NN \right\} \,.\]
\begin{remark}
  When $m\geq 0$, the first and the second eigenpairs of
  $\mathfrak{L}_{m}^{[0]}$ are respectively
  \[\left(\beta(0) ,t^{\frac{m}{2}} e^{\frac{-\beta(0)t}{2}} \right)
    \qquad \text{ and } \qquad \left(
      3\beta(0),\left[\beta(0)t-m-1\right]t^{\frac{m}{2}}
      e^{\frac{-\beta(0)t}{2}}\right).\]
\end{remark}
By using the $t^{\frac{m}{2}} e^{\frac{-\beta(0)t}{2}}$ as test
function and then the Spectral Theorem, we get the following (this
result will be recovered later by using a WKB analysis).
\begin{proposition}\label{Estimate Eig}
  For all $m\in \NN$, there exist $C>0$ and $h_0>0$ such that, for all
  $h\in (0,h_0)$,
  \begin{equation*}
    \textup{dist}\left(\mu_{m,0} h + \mu_{m,1}h^2  ,\textup{Sp}\left(\mathcal{M}_{h,m} \right)\right) \leq C h^{3} \,,
  \end{equation*}
  where $\mu_{m,0} = \beta(0)$ and
  $ \mu_{m,1} = \frac{(m+1)\beta'(0)}{\beta(0)}$.
\end{proposition}

\subsubsection{Semiclassical Agmon estimate}
The above proposition states that, for each $m\in \NN$, one can find
an eigenvalue of $\mathcal{M}_{h,m}$ near
$\beta(0) h+ \frac{(m+1) \beta'(0)}{\beta(0)} h^2$. We will use Agmon
estimates to prove that this eigenvalue is in fact
$\lambda_0(\mathcal{M}_{h,m})$.  In this section, let us denote
\[ \mathcal{V}_{h,m}(t)= \frac{(hm-a(ht))^2}{2ht}\,,\] and consider
the quadratic form associated with $\mathcal{M}_{h,m}$:
\[ Q_{h,m}(u) = 2h\int_{0}^{\infty} t \vert \partial_{t}u \vert^2
  \,\dd t + \int_{0}^{\infty} \mathcal{V}_{h,m}(t)\vert u(t)\vert^2
  \,\dd t \,.\]
Since we only deal with the Hilbert space $\textup{L}^2(\RR^{+},\dd t)$ in this section, we denote $\Vert \cdot \Vert_{\textup{L}^2}$ (resp. $\langle \cdot, \cdot\rangle_{\textup{L}^2}$) instead of $\Vert \cdot \Vert_{\textup{L}^2(\RR^{+},\dd t)}$ (resp. $\langle \cdot, \cdot\rangle_{\textup{L}^2(\RR^{+},\dd t)}$).\\
Let us recall that $a(\rho)=\int_{0}^{\rho} \beta(s) \dd s$. From
\eqref{QT EQ Condition Of Beta}, we get
\[a(ht) \geq \beta(0) ht \,. \] For $t$ large enough, we have
\begin{equation}\label{QT EQ Estimate V}
  \mathcal{V}_{h,m}(t)= \frac{(hm-a(ht))^2}{2ht} \geq h\frac{(\beta(0)t-m)^2}{2t}\,.
\end{equation}
\begin{proposition}\label{QT TH Lax R}
  Let $m\in \NN$ and let $\Phi \in W^{1,\infty}(\RR^{+},\RR)$, then
  for all $u\in \textup{Dom}(\mathcal{M}_{h,m})$, we have
  \begin{equation}\label{Agmon form}
    Q_{h,m}(e^{\Phi}u) = \textup{Re}\langle \mathcal{M}_{h,m}u, e^{2\Phi} u \rangle_{\textup{L}^2} +2h \int_0^{+\infty} t (\Phi '(t) )^2e^{2\Phi} \vert u \vert^2 \,dt\,.
  \end{equation}
\end{proposition}
\begin{proof}
  We have
  \begin{equation}\label{Integral by part eq}
    2h\langle \sqrt{t} \partial_{t} u, \sqrt{t} \partial_{t}(e^{2\Phi} u) \rangle_{\textup{L}^2}+\int_0^{+\infty} \mathcal{V}_{h,m}(t) e^{2\Phi} \vert u \vert^2 \,dt = \langle \mathcal{M}_{h,m} u,e^{2\Phi} u \rangle_{\textup{L}^2}\,.
  \end{equation}
  We set $P:= \sqrt{t}\partial_{t} $. Notice that the commutator
  \[ [P, e^{\Phi} ] = \sqrt{t} \Phi'(t) e^{\Phi}\] is a multiplication
  operator. We have
  \begin{eqnarray*}
    \langle P u, P e^{2\Phi} u \rangle_{\textup{L}^2} &=& \langle P u, [P, e^{\Phi} ] e^{\Phi} u\rangle_{\textup{L}^2} + \langle P u,  e^{\Phi} P e^{\Phi} u\rangle_{\textup{L}^2}\\
                                                      &=& \langle e^{\Phi} P u, [P, e^{\Phi} ]  u\rangle_{\textup{L}^2} + \langle e^{\Phi} P u,   P e^{\Phi} u\rangle_{\textup{L}^2} \\
                                                      &=& \langle e^{\Phi} P u, [P, e^{\Phi} ]  u\rangle_{\textup{L}^2} + \langle  P e^{\Phi} u,   P e^{\Phi} u\rangle_{\textup{L}^2} +  \langle  [e^{\Phi},P] u,   P e^{\Phi} u\rangle_{\textup{L}^2}\\
                                                      &=& \Vert  P e^{\Phi} u\Vert_{\textup{L}^2}^2 -\Vert [P,e^{\Phi}]u \Vert_{\textup{L}^2}^2+ \langle e^{\Phi} P u, [P, e^{\Phi} ]  u\rangle_{\textup{L}^2} - \langle [P, e^{\Phi} ]  u, e^{\Phi} P u \rangle_{\textup{L}^2}\,.
  \end{eqnarray*}
  Taking the real part of \eqref{Integral by part eq}, we get
  \begin{eqnarray*}
    && 2h \Vert  P e^{\Phi} u\Vert_{\textup{L}^2}^2 - 2h\Vert [P,e^{\Phi}]u \Vert_{\textup{L}^2}^2 +\int_0^{+\infty}\mathcal{ V}_{h,m}(t) e^{2\Phi} \vert u \vert^2 \,dt = \textup{Re} \langle \mathcal{M}_{h,m}u, e^{2\Phi} u \rangle_{\textup{L}^2}\,.
  \end{eqnarray*}
\end{proof}
\begin{proposition}\label{Agmon type 1}
  Under the assumption \eqref{QT EQ Condition Of Beta}, for all
  $m\in \NN$, and for all $\varepsilon \in (0,\frac{\beta(0)}{2})$,
  there exists $M>0$ such that
  \begin{equation*}
    \Vert e^{\varepsilon t}\Psi\Vert_{\textup{L}^2}^2 \leq M \Vert \Psi \Vert_{\textup{L}^2}^2\,,  \qquad \text{ and }\qquad  Q_{h,m}(e^{\varepsilon t}\Psi) \leq M\,h \Vert \Psi \Vert_{\textup{L}^2}^2\,,
  \end{equation*}
  for all eigenfunctions $\Psi$ of the operator $\mathcal{M}_{h,m}$,
  with eigenvalue of order $h$.
\end{proposition}
\begin{proof}
  Let us consider a sequence of functions
  $\left(\chi_{k}\right)_{k\geq 1}$ defined as follows
  \begin{equation}\label{eq.psi}
    \chi_k(t) =\left\{ 
      \begin{aligned}
        &t    \qquad &&\text{ for } 0\leq t\leq k,\\
        &2k-t \qquad &&\text{ for } k\leq t\leq 2k,\\
        &0    \qquad &&\text{ for } t\geq 2k.\\
      \end{aligned}\right.
  \end{equation}
  Notice that $\chi_k \in W^{1,\infty}(\RR^{+},\RR)$ and $\vert \chi_{k}'(t)\vert \leq 1$ for all $t\in \RR^{+}$.\\
  Let us consider the equation
  \begin{equation}
    \mathcal{M}_{h,m} \Psi = \lambda \Psi \qquad \text{ with } \quad \lambda \leq Ch\,.
  \end{equation}
  Using Proposition \ref{QT TH Lax R} with
  $\Phi =\varepsilon\chi_k(t)$, we get
  \begin{align*}
    Q_{h,m}(e^{\varepsilon\chi_k(t)}\Psi) &\leq \lambda \Vert e^{\varepsilon\chi_k(t)} \Psi\Vert_{\textup{L}^2}^2 +2h\varepsilon^2 \int_0^{+\infty} t \vert \chi_{k}'(t)\vert^2e^{2\varepsilon\chi_k(t)} \vert \Psi(t) \vert^2 \,dt\\
                                          &\leq Ch \Vert e^{\varepsilon\chi_k(t)} \Psi\Vert_{\textup{L}^2}^2 +2h\varepsilon^2 \int_0^{+\infty} t e^{2\varepsilon\chi_k(t)} \vert \Psi(t) \vert^2 \,dt\,.
  \end{align*}
  It implies that
  \begin{align*}
    \int_0^{+\infty}\mathcal{ V}_{h,m}(t) e^{2\varepsilon\chi_k(t)} \vert \Psi \vert^2 \,dt \leq Ch \Vert e^{\varepsilon\chi_k(t)} \Psi\Vert_{\textup{L}^2}^2 +2h\varepsilon^2 \int_0^{+\infty} t e^{2\varepsilon\chi_k(t)} \vert \Psi(t) \vert^2 \,dt\,,
  \end{align*}
  so that
  \begin{equation*}
    \int_0^{+\infty} \left(\mathcal{V}_{h,m}(t) -Ch -2h\varepsilon^2 t\right)e^{2\varepsilon\chi_k(t)} \vert \Psi \vert^2 \,dt \leq 0\,.
  \end{equation*}
  From \eqref{QT EQ Estimate V}, there exist $R>0$ and
  $C_1(R,\varepsilon)>0$ such that for all $t\geq R$, we have
  \begin{align}\label{QT EQ Vhm}
    \mathcal{V}_{h,m}(t) -Ch -2\varepsilon^2 th &\geq h \left[\frac{(b_0t-m)^2}{2t} -2\varepsilon^2t- C \right]\nonumber\\
                                                &= h \left[\left(\frac{b_0^2}{2}-2\varepsilon^2\right)t +\frac{m^2}{2t}-b_0m -C \right]\nonumber \\
                                                & \geq C_1(R,\varepsilon) h\,.
  \end{align}
  We deduce the existence of $C_2(R,\varepsilon)>0$ such that, for all
  $k\geq 1$,
  \begin{eqnarray*}
    C_1(R,\varepsilon) h \int_R^{+\infty} e^{2\varepsilon\chi_k(t)} \vert \Psi \vert^2 \,dt & \leq &   \int_{R}^{+\infty} \left(\mathcal{V}_{h,m}(t) -Ch -2h\varepsilon^2 t\right)e^{2\varepsilon\chi_k(t)} \vert \Psi \vert^2 \,dt \\
                                                                                            & \leq & \int_0^{R} \left( 2h\varepsilon^2 t+Ch-\mathcal{V}_{h,m}(t)\right)e^{2\varepsilon\chi_k(t)} \vert \Psi \vert^2 \,dt\\
                                                                                            & \leq & C_2(R,\varepsilon) h \Vert \Psi \Vert_{\textup{L}^2}^2\,.
  \end{eqnarray*}
  Then, there exists a constant $C(R,\varepsilon)$ such that, for all
  $k\geq1$,
  \begin{eqnarray*}
    \int_0^{+\infty} e^{2\varepsilon\chi_k(t)} \vert \Psi \vert^2 \,dt &=& \int_0^{R} e^{2\varepsilon\chi_k(t)} \vert \Psi \vert^2 \,dt +\int_R^{+\infty} e^{2\varepsilon\chi_k(t)} \vert \Psi \vert^2 \,dt \\
                                                                       & \leq & \int_0^{R} e^{2\varepsilon R} \vert \Psi \vert^2 \,dt +\int_R^{+\infty} e^{2\varepsilon\chi_k(t)} \vert \Psi \vert^2 \,dt \\
                                                                       & \leq &C(R,\varepsilon) \Vert \Psi \Vert_{\textup{L}^2}^2.
  \end{eqnarray*}
  Letting $k\to +\infty$ and using Fatou's lemma give
  \begin{equation*}
    \int_0^{+\infty} e^{2\varepsilon t} \vert \Psi \vert^2 \,dt \leq C(R,\varepsilon)  \Vert \Psi \Vert_{\textup{L}^2}^2.
  \end{equation*}
  Using again \eqref{QT EQ Vhm}, we notice that
  \[ \mathcal{V}_{h,m}(t) -Ch -2\varepsilon^2 th \geq
    C_{1}(R,\varepsilon) ht\,.\] Following the same steps as above, we
  get
  \begin{equation*}
    \int_0^{+\infty} t e^{2\varepsilon\chi_k(t)} \vert \Psi(t) \vert^2 \,dt \leq  C \Vert \Psi \Vert_{\textup{L}^2}^2.
  \end{equation*}
  Then, it leads to
  \begin{equation*}
    Q_{h,m}(e^{\varepsilon \chi_{k}(t)}\Psi) \leq Ch  \Vert \Psi \Vert_{\textup{L}^2}^2,
  \end{equation*}
  and we get the result.
\end{proof}
\begin{proposition}\label{Theo of lambda 0 of M}
  For all $m\in \NN$,
  \begin{equation}
    \lambda_{0}(\mathcal{M}_{h,m}) = \beta(0) h +  \frac{(m+1)\beta'(0)}{\beta(0)}h^2 +o(h^2)\,.
  \end{equation}
\end{proposition}
\begin{proof}
  Let us fix $m \in \NN$. We can choose the first eigenpairs
  $(\lambda_{i}(\mathcal{M}_{h,m}), \Psi_{i,h})_{i=1,2} $ such that
  $\Psi_{0,h}$ and $\Psi_{1,h}$ are orthogonal. We consider the
  two-dimensional space
  \begin{equation}
    E(h) = \text{ span} ( \Psi_{0,h}, \Psi_{1,h})\,.
  \end{equation}
  Let $\Psi \in E(h)$, we have
  \begin{eqnarray*}
    Q_{h,m}(\Psi) &=& 2h \int_0^{+\infty} t \vert \partial_{t} \left(\Psi \right) \vert^2 \, dt+ \int_0^{+\infty}\mathcal{ V}_{h,m}(t) \vert \Psi \vert^2 \,dt \\
                  &=& h \int_0^{+\infty} 2t \vert \partial_{t} \left(\Psi \right)\vert^2 \,+ \left( \frac{\beta(0)^2}{2}t + \frac{m^2}{2t} - m \beta(0)\right)\vert \Psi \vert^2 dt  \\
                  & & + \int_0^{+\infty} \left[\mathcal{ V}_{h,m}(t) -h\left( \frac{\beta(0)^2}{2}t + \frac{m^2}{2t} - m \beta(0)\right)\right]\vert \Psi \vert^2 \,dt\\
                  &\geq & h\mathcal{Q}_{m}^{[0]}(\Psi) - \int_0^{+\infty} \left\vert\mathcal{ V}_{h,m}(t) -h\left( \frac{\beta(0)^2}{2}t + \frac{m^2}{2t} - m \beta(0)\right)\right\vert \vert \Psi \vert^2 \,dt \,,\\
  \end{eqnarray*}
  where $\mathcal{Q}_{m}^{[0]}$ is the quadratic form associated with
  $\mathfrak{L}_{m}^{[0]}$.  We have
  \begin{equation*}
    \mathcal{V}_{h,m}(t) -h\left( \frac{\beta(0)^2}{2}t + \frac{m^2}{2t} - m \beta(0)\right) = \frac{\mathcal{O}((ht)^2)}{t}.
  \end{equation*}
  There exist $C_1>0$ and $\delta>0$ such that, for all $(h,t)$
  satisfying $ht\leq \delta$,
  \begin{equation*}
    \left\vert\mathcal{ V}_{h,m}(t) -h\left( \frac{\beta(0)^2}{2}t + \frac{m^2}{2t} - m \beta(0)\right)\right\vert \leq C_1 th^2 \,.
  \end{equation*}
  We have
  \begin{align*}
    \int_0^{\delta/h} \left\vert\mathcal{ V}_{h,m}(t) -h\left( \frac{\beta(0)^2}{2}t + \frac{m^2}{2t} - m \beta(0)\right)\right\vert  \vert \Psi \vert^2 \, dt &\leq  C_1h^2 \int_0^{\delta/h} t  \vert \Psi \vert^2  \,dt\\
                                                                                                                                                               &\leq C_2 h^2  \int_0^{+\infty} e^{2\varepsilon t}  \vert \Psi \vert^2  \,dt\\
                                                                                                                                                               &\leq Ch^2 \Vert \Psi \Vert_{\textup{L}^2}^2.
  \end{align*}
  Moreover,
  \begin{align*}
    \int_{\delta/h}^{+\infty} \left\vert\mathcal{ V}_{h,m}(t) \right\vert \vert \Psi \vert^2 \, dt &= \int_{\delta/h}^{+\infty} e^{-2\varepsilon t}  \left\vert\mathcal{ V}_{h,m}(t)\right\vert e^{2\varepsilon t}\vert \Psi \vert^2 \, dt\\
                                                                                                   &\leq \max_{t\geq \delta/h} e^{-2\varepsilon t}  \int_{0}^{+\infty} \left\vert\mathcal{ V}_{h,m}(t)\right\vert e^{2\varepsilon t}\vert \Psi \vert^2 \, dt\\
                                                                                                   &\leq Ch^2 \Vert \Psi \Vert_{\textup{L}^2}^2\,.
  \end{align*}
  Similarly, we also have
  \begin{align*}
    \int_{\delta/h}^{+\infty} h\left( \frac{\beta(0)^2}{2}t + \frac{m^2}{2t} - m \beta(0)\right) \,dt \leq Ch^2\Vert \Psi \Vert_{\textup{L}^2}^2 \,.
  \end{align*}
  Thus, we have the estimate
  \begin{equation*}
    Q_{h,m}(\Psi) \geq h\mathcal{Q}_{m}^{[0]}(\Psi) - Ch^2 \Vert \Psi \Vert_{\textup{L}^2}^2\,.
  \end{equation*}
  Since $\Psi \in E(h)$, there exists
  $(\alpha_1,\alpha_2)\in \mathbb{C}^2$ such that
  $\Psi = \alpha \Psi_{0,h}+\beta \Psi_{1,h}$ and
  \begin{equation*}
    Q_{h,m}(\Psi) = \lambda_{0}(\mathcal{M}_{h,m}) \vert \alpha_1 \vert^2 \Vert \Psi_{0,h} \Vert_{\textup{L}^2}^2+\lambda_{1}(\mathcal{M}_{h,m}) \vert \alpha_2\vert^2 \Vert \Psi_{1,h} \Vert_{\textup{L}^2}^2 \leq \lambda_{1}(\mathcal{M}_{h,m}) \Vert \Psi \Vert_{\textup{L}^2}^2.
  \end{equation*}
  By the min-max principle, we deduce that
  \begin{equation}\label{QT lambda 1 of M}
    \lambda_{1}(\mathcal{M}_{h,m}) \geq 3\beta(0)h -Ch^2 \,.
  \end{equation}
  From Proposition \ref{Estimate Eig}, we see that there exists an
  eigenvalue of $\mathcal{M}_{h,m}$ which is near
  $\beta(0)h+\frac{(m+1)\beta'(0)}{\beta(0)}h^2$ modulo $o(h^2)$. This
  eigenvalue can not be $\lambda_{k}(\mathcal{M}_{h,m})$ with
  $k\geq 1$, because if it were, the estimate
  \begin{equation*}
    3\beta(0)h -Ch^2 \leq \lambda_{k}(\mathcal{M}_{h,m}) = \beta(0)h+\frac{(m+1)\beta'(0)}{\beta(0)}h^2+o(h^2)\,,
  \end{equation*}
  would give a contradiction. The conclusion follows.
\end{proof}
\subsection{Eigenfunctions of the magnetic Laplacian and of its
  fibration}
We recall the expression of the magnetic Laplacian operator
$\mathscr{K}_{h}$
\begin{equation}
  \mathscr{K}_{h} =  -h^2 r^{-2} \left(r \partial_r \right)^2  \  + r^{-2}(-ih \partial_\theta - G(r))^2\,.
\end{equation}
For each $m\in \NN$, let $\lambda_0(\mathcal{N}_{h,m})$ be the first
eigenvalue of the operator $\mathcal{N}_{h,m}$ and $\Psi_{h,m}$ be the
associated eigenfunction.  Since $\mathcal{M}_{h,m}$ and
$\mathcal{N}_{h,m}$ are unitary equivalent, from Proposition \ref{Theo
  of lambda 0 of M}, we obtain
\begin{equation}\label{QT Eigen of N}
  \lambda_0(\mathcal{N}_{h,m})=\lambda_{0}(\mathcal{M}_{h,m})=\beta(0) h +  \frac{(m+1)\beta'(0)}{\beta(0)}h^2 +o(h^2)\,.
\end{equation}
On the other hand, $\mathfrak{L}_{h,m}$ and $\mathcal{N}_{h,m}$ also
are equivalent by
\[ \mathfrak{L}_{h,m} = T_1 \mathcal{N}_{h,m} T_1^{-1}\,,\] where
$T_1$ is defined in \eqref{QT EQ Unitary T1}. Therefore,
$\lambda_0(\mathcal{N}_{h,m})$ is also the first eigenvalue of
$\mathfrak{L}_{h,m}$ and $T_1(\Psi_{h,m})$ is the associated
eigenfunction. It results that
\begin{align*}
  \left[ \mathscr{K}_{h} -\lambda_0(\mathcal{N}_{h,m})\right]e^{im\theta}T_1(\Psi_{h,m}) = \left[ \mathfrak{L}_{h,m} -\lambda_0(\mathcal{M}_{h,m})\right] T_1(\Psi_{h,m}) e^{im\theta}=0\,.
\end{align*}
Thus $\lambda_0(\mathcal{N}_{h,m})$ belongs to the spectrum of
$\mathscr{K}_{h}$. Now, the result by Helffer and Kordyukov tells us
that, for all $k\in \NN$, the $k-$th eigenvalue of $\mathscr{K}_{h}$,
denoted by $\lambda_{k}(\mathscr{K}_{h})$ satisfies
\begin{equation}
  \lambda_{k}(\mathscr{K}_{h}) = \textbf{B}(0,0) h +  \left(2k\frac{\sqrt{\det H}}{\textbf{B}(0,0)}+\frac{(\text{Tr} H^{1/2})^2}{2\textbf{B}(0,0)}\right)h^2 +o(h^2),
\end{equation}
where $H= \frac{1}{2}\text{Hess}_{(0,0)}\textbf{B}$, so that
\begin{equation}
  \lambda_{k}(\mathscr{K}_{h}) = \beta(0) h +  \frac{(k+1)\beta'(0)}{\beta(0)}h^2 +o(h^2).
\end{equation}
Since $\lambda_{0}(\mathcal{N}_{h,m})$ is the eigenvalue of
$\mathscr{K}_{h}$, thus there exists $k\in \NN$ such that
\[\lambda_{0}(\mathcal{M}_{h,m})=\lambda_{k}(\mathscr{K}_{h})\,.\]
Considering \eqref{QT Eigen of N} and taking $h$ small enough, we
immediately obtain $k=m$ and
\[\lambda_{m}(\mathscr{K}_{h}) = \lambda_{0}(\mathcal{M}_{h,m})\,.\]
Since $\lambda_{m}(\mathscr{K}_h)$ is a simple eigenvalue, we have the
following statement
\begin{proposition}\label{QT Relation}
  When $h$ is small enough, the $m$-th eigenvalue of the magnetic
  Laplacian $\mathscr{K}_h$ is exactly the first eigenvalue of the
  operator $\mathcal{N}_{h,m}$:
  \[\lambda_{m}(\mathscr{K}_h) = \lambda_0( \mathcal{N}_{h,m})\,.\]
  The $m$-th eigenfunction of $\mathscr{K}_h$ is in the form
  \begin{equation}
    c e^{im\theta} \Psi_{h,m}\left(\frac{r^2}{2}\right)\,,
  \end{equation}
  where $\Psi_{h,m}$ is a ground-state of the operator
  $\mathcal{N}_{h,m}$ and $c\in \mathbb{C}\setminus \{0\}$.
\end{proposition}
\subsection{Radial magnetic WKB construction}\label{subsection radial
  WKB}
In this section, we focus on constructing a WKB Ansatz. Thanks to
Proposition \ref{QT Relation}, we just need to do the WKB analysis for
the operator $\mathcal{N}_{h,m}$.  Since $\mathcal{N}_{h,m}$ is a real
electric Schr\"{o}dinger operator in dimension $1$, one can easily
find a WKB approximation of $\Psi_{h,m}$. We recall that
\begin{equation}
  \mathcal{N}_{h,m} = -2h^2 \partial_{\rho}\rho\partial_{\rho} + \frac{(hm-a(\rho))^2}{2\rho}\,,\quad a(\rho) = \int_{0}^{\rho} \beta(\tau) \, d\tau\,.
\end{equation}
We consider the conjugated operator with real-valued smooth function
$\varphi$:
\begin{eqnarray*}
  \widehat{\mathcal{N}_{h,m}} = e^{\frac{\varphi(\rho)}{h}}\rho^{\frac{-m}{2}}\, \mathcal{N}_{h,m} \, \rho^{\frac{m}{2}}e^{\frac{-\varphi(\rho)}{h}}\,,
\end{eqnarray*}
and get
\begin{eqnarray*}
  \widehat{\mathcal{N}_{h,m}} &=& \left[\frac{(a(\rho))^2}{2\rho}-2\rho (\varphi'(\rho))^2 \right]+ h\left[4\varphi'(\rho)\rho\partial_{\rho}+ 2\varphi'(\rho)+ 2\rho\varphi''(\rho)+ 2m \varphi'(\rho)-m\frac{a(\rho)}{\rho}\right] \\
                              & & + h^2 \left[ -2\rho\partial^2_{\rho} -(2m+2)\partial_{\rho}  \right]\,.
\end{eqnarray*}
\subsubsection{The eikonal equation}
The eikonal equation reads
\begin{equation}\label{QT eikonal}
  (\varphi'(\rho))^2 = \frac{(a(\rho))^2}{4\rho^2}\,.
\end{equation}
We choose a positive solution
\[ \varphi(\rho) = \int_{0}^{\rho} \frac{a(\tau)}{2\tau} \, \dd\tau,\]
which is a smooth function on $[0,+\infty)$ because it can be
rewritten in the form
\[\varphi(\rho)
  = \int_{0}^{\rho} \frac{1}{2\tau} \int_{0}^{\tau} \beta(\xi)\, \dd
  \xi\dd \tau = \frac{1}{2} \int_{0}^{\rho} \int_{0}^{1} \beta(\xi
  \tau)\, \dd \xi \dd \tau \,.\] Then the operator $\mathcal{N}_{h,m}$
becomes
\begin{eqnarray*}
  \widehat{\mathcal{N}_{h,m}} &=& h\mathcal{N}^1 + h^2\mathcal{N}^2_{m} \,,
\end{eqnarray*}
where
\begin{equation*}
  \mathcal{N}^1 = 4\varphi'(\rho)\rho\partial_{\rho}+ 2\varphi'(\rho)+ 2\rho\varphi''(\rho)=2 a(\rho) \partial_{\rho}+ \beta(\rho)\,,
\end{equation*}
and
\begin{equation*}
  \mathcal{N}^2_{m} =-2\rho\partial^2_{\rho} -(2m+2)\partial_{\rho}\,.
\end{equation*}
We now look for a WKB Ansatz and a quasi-eigenvalue
\begin{align*}
  &a(\rho,h) \sim a_0(\rho) + h a_1(\rho) + h^2 a_2(\rho)+ ...\,,\\
  &\lambda(h) \sim h(\mu_0 +h\mu_1 +h^2\mu_2+...) \,.
\end{align*}
We substitute these formal series into the equation
\[ \left(\widehat{\mathcal{N}_{h,m}}-\lambda(h) \right)
  a(\rho,h)=0\,, \] and get
\begin{align*}
  &h :   & \left(\mathcal{N}^1 -\mu_0 \right)a_0 &= 0\\
  &h^2 : & \left(\mathcal{N}^1 -\mu_0 \right)a_1 &= \left(\mu_1-\mathcal{N}^{2}_{m}  \right)a_0\\
  &...
\end{align*}
\subsubsection{The first transport equation}
Collecting all terms of order $h^1$, we have the first transport
equation
\begin{equation}\label{Eq transport 1 raidal}
  (2 a(\rho) \partial_{\rho}+ \beta(\rho)-\mu_{0}) a_{0} =0\,.
\end{equation}
This equation has smooth solutions which do not vanish at $0$ if and
only if
\[ \mu_{0}= \beta(0)\,.\] Indeed, since $a(0)=0$, the function
\begin{equation*}
  F(\rho):=\frac{\beta(0)-\beta(\rho)}{2a(\rho)}\,,
\end{equation*}
is actually smooth on $[0,+\infty)$. Then,
\begin{equation*}
  a_{0}(\rho) = a_{0}(0) \exp\left( \int_0^\rho F(s) \, \dd s \right)\,,
\end{equation*}
with $a_{0}(0) \neq 0$. We choose $a(0)=1$.

\subsubsection{The second transport equation}
Let us gather all terms of order $h^2$ to get the second transport
equation
\begin{equation}\label{Eq transport 2 radial}
  (2 a(\rho) \partial_{\rho}+ \beta(\rho)-\mu_{0}) a_{1} = (\mu_1 +(2m+2)\partial_{\rho}+ 2\rho\partial^2_{\rho} ) a_0\,.
\end{equation}
Thus, \eqref{Eq transport 2 radial} has a smooth solution if and only
if
\begin{equation*}
  (\mu_1 +(2m+2)\partial_{\rho} )a_0 (0)=0\,,
\end{equation*}
or
\begin{equation*}
  \mu_1 = -\frac{(2m+2)\partial_{\rho} a_{0}(0)}{a_{0}(0)}.
\end{equation*}
From \eqref{Eq transport 1 raidal}, we get
\begin{equation*}
  \frac{\partial_{\rho} a_{0}(0)}{a_{0}(0)} = \lim_{\rho \to 0}\frac{\beta(0)-\beta(\rho)}{2a(\rho)} = \frac{-\beta'(0)}{2\beta(0)}.
\end{equation*}
Thus,
\begin{equation}
  \mu_1 = (m+1) \frac{\beta'(0)}{\beta(0)}.
\end{equation}
With this choice, \eqref{Eq transport 2 radial} can be rewritten as
\begin{equation}
  \partial_{\rho}a_{1}- F(\rho)a_{1} = g_1(\rho):= \frac{(\mu_1 +(2m+2)\partial_{\rho}+ 2\rho\partial^2_{\rho} ) a_0 }{2a(\rho)}\,.
\end{equation}
This equation has solutions in the form
\begin{align*}
  a_{1}(\rho) &=  \exp\left( \int_0^\rho F(s) \, \dd s\right) \int_{0}^{\rho} \exp\left( -\int_0^{\tau} F(s) \, \dd s\right)g_1(\tau)\,\dd \tau  \\
              &\qquad+ a_1(0) \exp\left( \int_0^\rho F(s) \, \dd s\right) \,.
\end{align*}
We impose $a_1(0)=0$ so that
\[ a_{1}(\rho) = \exp\left( \int_0^\rho F(s) \, \dd s\right)
  \int_{0}^{\rho} \exp\left( -\int_0^{\tau} F(s) \, \dd
    s\right)g_1(\tau)\, \dd \tau\,. \]
\subsubsection{Induction}
Let $n \in \NN$ and $n\geq 2$. We assume that
$(\mu_j)_{0\leq j\leq n}$ and $(a_j)_{0\leq j\leq n} $ are determined
and $ (a_j)_{1\leq j\leq n}$ are smooth function on $[0,+\infty) $ and
vanish at $\rho=0$. Let us show that we can determine $\mu_{n+1}$ and
$a_{n+1}$ by the $(n+1)$-th transport equation
\begin{equation}\label{Eq Transport n+1 radial}
  \left(2 a(\rho) \partial_{\rho}+ \beta(\rho)-\mu_{0} \right)a_{n+1} = \left((2m+2)\partial_{\rho} +2\rho\partial_{\rho}^2 \right) a_n + \sum_{j=1}^{n+1} \mu_{j} a_{n+1-j} \,.
\end{equation}
The equation has a smooth solution at $0$ if and only if
\[ (2m+2)\partial_{\rho}a_n(0) + \sum_{j=1}^{n} \mu_{j}
  a_{n+1-j}(0)+\mu_{n+1} a_0(0) = 0. \] Since $a_{0}(0)=1$,
$\mu_{n+1}$ is completely determined by
\[\mu_{n+1} = -(2m+2) \partial_{\rho}a_n(0)\,. \]
With this value of $\mu_{n+1}$, we can rewrite the equation \eqref{Eq
  Transport n+1 radial} as
\begin{equation}\label{Transport (n+1)'}
  \partial_{\rho}a_{n+1}- F(\rho)a_{n+1} = g_{n}(\rho)\,,
\end{equation}
where $g_{n}$ is the smooth extension of the function
\begin{equation*}
  G_n(\rho) = \frac{\left((2m+2)\partial_{\rho} +2\rho\partial_{\rho}^2 \right) a_n + \sum_{j=1}^{n+1} \mu_{j} a_{n+1-j} }{2a(\rho)}\,,
\end{equation*}
on $[0,+\infty)$.\\
There is only one solution $a_{n+1}$ such that $a_{n+1}(0)=0 $, that
is
\[ a_{n+1}(\rho) = \exp\left( \int_0^\rho F(s) \, \dd s\right)
  \int_{0}^{\rho} \exp\left( -\int_0^{\tau} F(s) \, \dd
    s\right)g_n(\tau)\, \dd \tau \,.\]
\begin{proof}[Proof of Theorem \ref{IN TH WKB R}]
  We fix $m\in \NN$. The WKB analysis provided us with functions and sequences as follows:
  \begin{enumerate}[\rm (i)]
  \item The function $\varphi(\rho)$ is given by \eqref{QT eikonal}:
    \begin{equation}\label{QT EQ Phase}
      \varphi(\rho) = \frac{1}{2} \int_{0}^{\rho} \int_{0}^{1}  \beta(\xi \tau)\, \dd \xi \dd \tau\,.
    \end{equation}
  \item The transport equations give us the existence of a sequence of
    smooth functions $(a_{m,j})_{j\in\NN}$ defined on $[0,+\infty)$
    and the sequence $(\mu_{m,j})_{j\in\NN}$ which depends on
    $m$. Notice that $a_{m,0}$ is positive since
    \[ a_{m,0}(\rho) = \exp\left(\int_{0}^{\rho} F(s) \dd s
      \right)\,. \]
  \end{enumerate}
  For each $J\in \NN$, from the WKB construction, there exists a
  smooth function $f_{m,J}(\rho)$ defined on $[0,+\infty)$ such that
  \begin{align*}
    e^{\frac{\varphi(\rho)}{h}}\rho^{\frac{-m}{2}}\, \left(\mathcal{N}_{h,m} -h\sum_{j=0}^{J}\mu_{m,j}h^{j}\right)\, \left( \rho^{\frac{m}{2}}e^{\frac{-\varphi(\rho)}{h}} \sum_{j=0}^{J} a_{m,j} h^{j} \right) = f_{m,J}(\rho) h^{J+2} \,.
  \end{align*}
  After changing of variable $\rho=\frac{r^2}{2}$, we obtain
  \begin{align*}
    e^{\frac{\varphi\left(\frac{r^2}{2}\right)}{h}}\left(\frac{r^2}{2}\right)^{\frac{-m}{2}}\, \left(\mathfrak{L}_{h,m} -h\sum_{j=0}^{J}\mu_{m,j}h^{j}\right)\,& \left( \left(\frac{r^2}{2}\right)^{\frac{m}{2}}e^{\frac{-\varphi\left(\frac{r^2}{2}\right)}{h}} \sum_{j=0}^{J} a_{m,j}\left(\frac{r^2}{2}\right) h^{j} \right) \\
                                                                                                                                                               &= f_{m,J}\left(\frac{r^2}{2}\right) h^{J+2} \,.
  \end{align*}
  By multiplying
  $\displaystyle
  \left(\frac{r^2}{2}\right)^{\frac{m}{2}}e^{\frac{-\varphi\left(\frac{r^2}{2}\right)}{h}}
  \sum_{j=0}^{J} a_{m,j}\left(\frac{r^2}{2}\right) h^{j}$ with
  $e^{im\theta}$ and using the fact that
  \[ \mathscr{K}_{h} (e^{im\theta} u) = \mathfrak{L}_{h,m}
    (e^{im\theta} u)\,,\] we deduce Theorem \ref{IN TH WKB R} is
  easily deduced.
\end{proof}
\begin{corollary}\label{QT TH Ex R}
  For all $(\varepsilon, m, J)\in (0,1)\times \NN \times \NN$, there
  exist a constant $C>0$ and $h_0>0$ such that, for all
  $h\in (0,h_0)$,
  \begin{equation}\label{exponential inequality}
    \left\Vert e^{\varepsilon\varphi(\rho)/h}\left(\mathcal{N}_{h,m}-\lambda^{J}_{h,m} \right)\Psi_{h,m}^{J} \right\Vert_{\textup{L}^2(\RR^{+})} \leq C h^{J+2},
  \end{equation}
  where
  \begin{equation}\label{WKB eigen}
    \lambda^{J}_{h,m}:= h\sum_{j=0}^{J} \mu_{j,m}h^{j} \qquad \text{ and } \qquad \Psi_{h,m}^{J} (\rho) := \chi e^{-\varphi(\rho)/h}\rho^{\frac{m}{2}} \left(\sum_{j=0}^{J} a_{j,m}h^{j}\right)\,,
  \end{equation}
  where $\chi$ is defined in \eqref{IN EQ Cut-off}.
	
  In particular,
  \begin{equation}\label{non-exponential inequality}
    \left\Vert \left(\mathcal{N}_{h,m}-\lambda^{J}_{h,m} \right)\Psi_{h,m}^{J} \right\Vert_{\textup{L}^2(\RR^{+})} \leq C h^{J+2}\,.
  \end{equation}
\end{corollary}
We may now provide an approximation of the ground-state eigenfunction
of the operator $\mathcal{N}_{h,m}$ by the WKB construction
$\Psi_{h,m}^{J}$ defined in \eqref{WKB eigen}. Let $\Psi_{h,m}$ be an
eigenfunction according to $\lambda_0(\mathcal{N}_{h,m})$, we
introduce the orthogonal projection of $\Psi_{h,m}^{J} $ on the
eigenspace of $\lambda_0(\mathcal{N}_{h,m})$
\[
  \Gamma_{m} \Psi_{h,m}^{J} = \langle\Psi_{h,m}^{J} ,\Psi_{h,m}
  \rangle \Psi_{h,m}\,.
\]

\begin{corollary}\label{Theorem approximate 1 radial}
  For all $(m, J)\in \NN\times \NN$, there exist $C>0$ and $h_0>0$
  such that, for all $h\in(0,h_0)$,
  \begin{equation}\label{QT EQ Ap R}
    \left\Vert \Psi_{h,m}^{J} - \Gamma_{m} \Psi_{h,m}^{J} \right\Vert_{\textup{L}^2(\RR^{+})} \leq Ch^{J+1}.
  \end{equation}
\end{corollary}

\subsection{A stronger WKB approximation}
Let us recall the expression of the operator $\mathcal{N}_{h,m}$
\[\mathcal{N}_{h,m}= -2h^2 \partial_{\rho} \rho \partial_{\rho} +
  \widetilde{V}_{h,m}(\rho)\,,\] where
\[ \widetilde{V}_{h,m}(\rho):=\frac{\left(
      hm-a(\rho)\right)^2}{2\rho}\,.\]
\begin{proposition}\label{Prop Agmon radial}
  Let $m\in \NN$ and let
  $(\Phi_{k})_{k\in \NN} \subset W^{1,\infty}(\RR^+,\RR)$. Assume that
  there exist $M>0$, $K_1>0$, $K_2>0$ and $R_0>0$ such that for all
  $h\in(0,1)$, $k\in \NN$
  \begin{eqnarray}
    &\widetilde{V}_{h,m}(\rho)-2\rho \vert \Phi_{k}'(\rho)\vert^2 \geq  M h \qquad &\text{ for all } \rho \in [R_0h, +\infty)\label{out the ball}\,,\\
    &\vert \Phi_{k}'(\rho) \vert \leq K_1, \qquad \vert \Phi_{k}(\rho) \vert \leq  K_2 h \qquad &\text{ for all } \rho \in [0,R_0h)\label{in the ball}\,.
  \end{eqnarray}
  Then, for all $c_0\in (0,M)$, there exists a positive constant $C>0$
  such that, for all $h\in(0,1)$, $k\in \NN$, $z\in [0,c_0 h]$, and
  $u\in \textup{Dom}(\mathcal{N}_{h,m})$,
  \begin{equation}\label{Eq bound above radial}
    \Vert  e^{\Phi_{k}/h} u \Vert_{\textup{L}^2(\RR^{+})}  \leq \frac{C}{h}\Vert e^{\Phi_{k}/h} \left(\mathcal{N}_{h,m}-z\right)u \Vert_{\textup{L}^2(\RR^{+})} + C \Vert u \Vert_{\textup{L}^2(\RR^{+})}\,.
  \end{equation}
\end{proposition}
\begin{proof}
  We have
  \begin{equation}\label{QT EQ Lax}
    \left\langle \mathcal{N}_{h,m} u, e^{2\Phi_{k}/h}u \right\rangle_{\textup{L}^2(\RR^{+})}= 2h^2  \left\langle \sqrt{\rho} \partial_{\rho} u, \sqrt{\rho}\partial_{\rho}(e^{2\Phi_{k}/h}u)\right\rangle_{\textup{L}^2(\RR^{+})}
    +\int_{0}^{\infty} \widetilde{V}_{h,m}(\rho) e^{2\Phi_{k}/h}\vert u \vert^2 \, \dd \rho\,.
  \end{equation}
  Setting $P=\sqrt{\rho} \partial_{\rho}$, we get
  \[\textup{Re} \left(\left\langle P u, P e^{2\Phi_{k}/h}
        u\right\rangle_{\textup{L}^2(\RR^{+})} \right)= \Vert P
    e^{\Phi_{k}/h} u \Vert_{\textup{L}^2(\RR^{+})}^2 - \Vert [P,
    e^{\Phi_{k}/h}] u \Vert_{\textup{L}^2(\RR^{+})}^2\,.\] Noticing
  that
  $ [P, e^{\Phi_{k}/h}] = \frac{\sqrt{\rho}\Phi_{k}'}{h}
  e^{\Phi_{k}/h}$ and \eqref{QT EQ Lax} gives
  \begin{multline*}
    \mathrm{Re} \left\langle \mathcal{N}_{h,m}u, e^{2\Phi_{k}/h}u
    \right\rangle
    = 2h^2\int_{0}^{+\infty} \rho \vert \partial_{\rho}( e^{\Phi_{k}/h}u )\vert^2 \,\dd\rho\\
    +\int_{0}^{+\infty} \left(\widetilde{V}_{h,m}-2\rho \vert
      \Phi_{k}'(\rho) \vert^2\right) e^{2\Phi_{k}/h} \vert u \vert^2
    \,\dd \rho\,.
  \end{multline*}
  Since $\widetilde{V}_{h,m}(\rho) \geq 0 $, we get
  \begin{align*}
    \int_{R_{0}h}^{+\infty} \left(\widetilde{V}_{h,m}-2\rho \vert \Phi_{k}'(\rho) \vert^2\right)  \vert  e^{\Phi_{k}/h} u \vert^2 \,\dd \rho \leq &\Vert e^{\Phi_{k}/h} \mathcal{N}_{h,m}u \Vert_{\textup{L}^2(\RR^{+})} \Vert e^{\Phi_{k}/h} u  \Vert_{\textup{L}^2(\RR^{+})}\\
                                                                                                                                                  &+\int_{0}^{R_{0}h} 2\rho \vert \Phi_{k}'(\rho) \vert^2 e^{2\Phi_{k}/h} \vert u \vert^2 \,\dd \rho  \,.
  \end{align*}
  Using \eqref{out the ball}, we deduce that
  \begin{align*}
    Mh\int_{R_{0}h}^{+\infty} \vert  e^{\Phi_{k}/h} u \vert^2 \,\dd \rho \leq \Vert e^{\Phi_{k}/h} \mathcal{N}_{h,m}u \Vert_{\textup{L}^2(\RR^{+})} &\Vert e^{\Phi_{k}/h} u  \Vert_{\textup{L}^2(\RR^{+})} \\
                                                                                                                                                    &+\int_{0}^{R_{0}h} 2\rho \vert \Phi_{k}'(\rho) \vert^2 e^{2\Phi_{k}/h} \vert u \vert^2 \,\dd \rho \,.
  \end{align*}
  Thanks to \eqref{in the ball}, $\Phi_{k}/h$ and $\Phi_{k}'$ are
  uniformly bounded with respect to $h$ and to $k$ on
  $[0,R_0h)$. Therefore, there exists a constant $L>0$ (independent of
  $h$ and $k$) such that
  \[ Mh\int_{0}^{+\infty} \vert e^{\Phi_{k}/h} u \vert^2 \,\dd \rho
    \leq \Vert e^{\Phi_{k}/h} \mathcal{N}_{h,m}u
    \Vert_{\textup{L}^2(\RR^{+})}\Vert e^{\Phi_{k}/h} u
    \Vert_{\textup{L}^2(\RR^{+})} + Lh\int_{0}^{R_{0}h} \vert u
    \vert^2 \,\dd \rho \,.\] For $z\in [0,c_0h)$, we get
  \[ (M-c_0)h\Vert e^{\Phi_{k}/h} u \Vert^2_{\textup{L}^2(\RR^{+})}
    \leq \Vert e^{\Phi_{k}/h} \left(\mathcal{N}_{h,m}-z\right)u
    \Vert_{\textup{L}^2(\RR^{+})} \Vert e^{\Phi_{k}/h} u
    \Vert_{\textup{L}^2(\RR^{+})} + Lh \Vert u
    \Vert^2_{\textup{L}^2(\RR^{+})} \,.\] Since $M>c_0$, this gives
  \eqref{Eq bound above radial} .
\end{proof}
The first application of the above Agmon estimate is to prove the
decay of the eigenfunctions.
\begin{theorem}\label{QT TH L2}
  For all $\varepsilon\in (0,1)$, there exist $C>0$ and $h_0>0$ such
  that, for all $h\in (0,h_0)$ and all eigenfunctions $\Psi$ with
  eigenvalue of order $h$ of the operator $\mathcal{N}_{h,m}$,
  \begin{equation}
    \Vert e^{\varepsilon \varphi/h} \Psi \Vert_{\textup{L}^2(\RR^{+})} \leq C \Vert   \Psi \Vert_{\textup{L}^2(\RR^{+})}\,,
  \end{equation}
  where
  $\displaystyle \varphi(\rho) = \int_0^{\rho} \frac{a(\tau)}{2\tau}
  \, d\tau$ is given by \eqref{QT EQ Phase}.
\end{theorem}
\begin{proof}
  Let $(\chi_{k})_{k\in \NN}$ be a sequence of functions as in the
  proof of Proposition \eqref{Agmon type 1}. In order to apply
  Proposition \ref{Prop Agmon radial}, we consider
  \[ \Phi_{k}(\rho) = \varepsilon \chi_{k} (\varphi(\rho))\,. \] For
  each $k\in \NN$, we have $\Phi_{k} \in
  W^{1,\infty}(\RR^{+},\RR)$. Furthermore, one has
  \[ \vert \Phi_{k}'(\rho)\vert \leq \varepsilon \vert \varphi'(\rho)
    \vert=\frac{\varepsilon a(\rho)}{2\rho} \qquad \text{ a.e. on
    }\RR_{+}\,.\]

  Let us consider an eigenvalue $\lambda =(\mathcal{O}(h))$ and an
  associated eigenfunction $\Psi$. Then, there exist $c_0>0$ and
  $h_0>0$ such that
  \[ \vert \lambda \vert \leq c_0 h \qquad\text{ for all } h\in
    (0,h_0)\,.\] Let $M$ and $R_0$ be numbers such that
  \begin{equation*}
    \left\{ \begin{aligned}
        &M >c_0\,,\\
        &R_0 \geq  \frac{2\beta(0) m +2M}{\beta(0)^2(1-\varepsilon^2)}\,.
      \end{aligned}
    \right.
  \end{equation*}   
  Using \eqref{QT EQ Condition Of Beta}, we have
  $a(\rho) \geq b_0 \rho$ for all $\rho \in \RR_{+}$. From the
  definition of $R_0$, we have the estimate, for all $h\in (0,h_0)$,
  $k\in \NN$ and $\rho\geq R_0h$,
  \begin{align*}
    \widetilde{V}_{h,m}(\rho)- 2\rho \vert \Phi_{k}'(\rho) \vert^2 & \geq \widetilde{V}_{h,m}(\rho)  -\varepsilon^2 \frac{a^2(\rho)}{2\rho}\\
                                                                   &\geq   \frac{(1-\varepsilon^2)\left(\beta(0) \rho-\frac{hm}{1-\varepsilon^2} \right)^2 -\frac{h^2m^2}{1-\varepsilon^2}}{2\rho}\\
                                                                   &\geq \left( \frac{(1-\varepsilon^2)\beta(0)^2 R_0}{2} -\beta(0)m\right)h\\
                                                                   &\geq Mh .
  \end{align*}
  On the other hand, there exist $K_1>0$ and $K_2>0$ such that, for
  all $h\in (0,h_0)$, $k\in \NN$ and $\rho\in [0,R_0h)$,
  \[ \vert \Phi_{k}'(\rho) \vert \leq \frac{\varepsilon
      a(\rho)}{2\rho} = \frac{\varepsilon}{2} \int_{0}^{1} \beta(\rho
    s) \dd s \leq K_1\] and
  \[ \vert \Phi_k(\rho)\vert \leq \varepsilon \phi(\rho) = \varepsilon
    \int_{0}^{\rho} \int_{0}^{1} \beta(\tau s)\, \dd s\, \dd \tau \leq
    K_2 h\,. \]

  Now, we can apply Proposition \ref{Prop Agmon radial} for
  $z=\lambda$, there exists a constant $C>0$ such that, for all
  eigenfunction $\Psi$ associated with $\lambda$,
  \begin{equation*}
    \int_0^{+\infty}  e^{2\varepsilon \chi_{k} (\varphi/h)}  \vert \Psi \vert^2 \, \dd \rho \leq C \int_0^{+\infty}  \vert \Psi \vert^2 \,\dd \rho \,.
  \end{equation*}
  By letting $k\to \infty$ and using Fatou's lemma, we get
  \begin{equation*}
    \int_0^{+\infty}  e^{2\varepsilon \varphi/h }  \vert \Psi \vert^2 \, \dd \rho \leq C\int_0^{+\infty} \vert \Psi \vert^2 \, \dd \rho \,.
  \end{equation*}
\end{proof}

\begin{proof}[Proof of Theorem \ref{IN TH L2}]
  Let
  $ T: \textup{L}^2(\RR^2,\dd q) \to \textup{L}^2(\RR^{+}\times \RR/2
  \pi\mathbb{Z} ,r \dd r \dd \theta)$ be the unitary operator
  associated with the polar coordinates. Then, $T(U_{h,m})$ is the
  eigenfunction associated with the eigenvalue
  $\lambda_{m}(\mathscr{K}_{h})=\lambda_{m}(\mathscr{L}_{h,\A})$ of
  the operator $\mathscr{K}_{h}$. From Proposition \ref{QT Relation},
  \[T(U_{h,m}) = \frac{1}{\sqrt{2\pi}} e^{-im\theta} \Psi_{h,m}\left(
      \frac{r^2}{2}\right) \,,\] where $\Psi_{h,m}$ is a eigenfunction
  associated with the first eigenvalue
  $\lambda_{0}(\mathcal{N}_{h,m})$ of the operator
  $\mathcal{N}_{h,m}$. Theorem \ref{IN TH L2} easily follows.
\end{proof}

\begin{theorem}\label{QT TH Ex Est R}
For all $(\varepsilon, m, J)\in (0,1)\times \NN \times \NN $, there exist $C>0$ and $h_0>0$ such that, for all $h\in(0,h_0)$,
\begin{equation}\label{QT EQ Ex Est R}
 \Vert e^{\varepsilon\varphi(\rho)/h} \left( \Psi_{h,m}^{J} - \Gamma_{m} \Psi_{h,m}^{J} \right)\Vert_{\textup{L}^2(\RR^{+})} \leq Ch^{J+1},
\end{equation}
where $\displaystyle \varphi(\rho) = \int_0^{\rho} \frac{a(\tau)}{2\tau} \, d\tau$ is given by \eqref{QT EQ Phase}.
\end{theorem}
\begin{proof}
Let us fix $(\varepsilon, m, J)\in (0,1)\times \NN \times \NN $. We recall that $\Gamma_{m} \Psi_{h,m}^{J}$ is the eigenfunction with the eigenvalue $\lambda_0(\mathcal{N}_{h,m})$ that has order $h$. We consider again the sequence $(\Phi_k)_{k\in\NN}$. Applying Proposition \ref{Prop Agmon radial} to $u=\Psi_{h,m}^{J} - \Gamma_{m} \Psi_{h,m}^{J}$ we find
\begin{equation}\label{Inequality radial 2}
\Vert  e^{\Phi_{k}/h} u \Vert  \leq \frac{C}{h}\Vert e^{\Phi_{k}/h} \left(\mathcal{N}_{h,m}-\lambda_0(\mathcal{N}_{h,m})\right)u \Vert + C \Vert u\Vert\,.
\end{equation}
Thanks to \eqref{exponential inequality}, we have 
\begin{eqnarray*}
&\,& \left\Vert e^{\Phi_{k}/h} \left(\mathcal{N}_{h,m}-\lambda_0(\mathcal{N}_{h,m}) \right)u\right\Vert_{\textup{L}^2(\RR^{+})}\\
&\leq & \left\Vert e^{\Phi_{k}/h} \left(\mathcal{N}_{h,m}-\lambda_{h,m}^{J} \right)\Psi_{h,m}^{J} \right\Vert_{\textup{L}^2(\RR^{+})} + \left\vert \lambda_0(\mathcal{N}_{h,m})-\lambda^{J}_{h,m} \right\vert \left\Vert e^{\Phi_{k}/h} \Psi_{h,m}^{J} \right\Vert_{\textup{L}^2(\RR^{+})}\\
&\leq & Ch^{J+2}+C h^{J+1}\left\vert \lambda_0(\mathcal{N}_{h,m})-\lambda^{J}_{h,m}\right\vert+ C\vert \lambda_0(\mathcal{N}_{h,m})-\lambda^{J}_{h,m} \vert \left\Vert \Psi_{h,m}^{J} \right\Vert_{\textup{L}^2(\RR^{+})}\\
&\leq &  Ch^{J+2}\,.
\end{eqnarray*}
Using Corollary \ref{Theorem approximate 1 radial} and \eqref{Inequality radial 2}, we get
\[\Vert e^{\varepsilon \chi_{k}(\varphi(\rho))/h} u \Vert_{\textup{L}^2(\RR^{+})} \leq Ch^{J+1}\,, \]
for all $k\geq 1$.
Then, we take to limit $k\to +\infty$ and use Fatou's lemma.
\end{proof}

\appendix
\section{Spectrum of the Laguerre operator}\label{Appendix A}
This appendix is devoted to the Laguerre operators
\begin{eqnarray*}
\mathcal{T}_{m}= -2s \partial_{s}^2 +\left(2s-2-2\vert m \vert \right)\partial_{s} +\vert m \vert -m+1 \,,
\end{eqnarray*}
for each $m\in \ZZ$.
We denote by $L^{( m)}_n $ the generalized Laguerre polynomials: these are solutions of the differential equation
\begin{equation}
s \partial_{s}^2y +\left(\vert m \vert+1 -s \right)\partial_{s}y +ny=0\,,
\end{equation}
with $ n\in \ZZ$ , see \cite{S75}. Then, for each $m \in \ZZ$, we have
\begin{equation}
\mathcal{T}_{m}(L^{( m)}_n) = (2n +1+\vert m\vert -m )L^{( m)}_n .
\end{equation}
In particular, 
\begin{equation}\label{QT Spectrum of Laguerre 1}
\left\{2n +\vert m\vert -m +1 :n\in \NN\right\} \subset \mathrm{sp}(\mathcal{T}_{m})\,.
\end{equation}
These polynomials are orthogonal with the inner product of space $ \textup{L}^2(\RR^{+},s^{\vert m \vert} e^{-s} ds)$ and satisfy
\begin{equation*}
\int_0^{+\infty} L^{( m)}_k(s) L^{( m)}_n(s) \,s^{\vert m \vert} e^{-s} ds= \frac{\Gamma(n+\vert m\vert +1)}{n!} \delta_{k,n}\,,
\end{equation*}
where $\delta_{k,n}$ denotes Kronecker symbol.
\begin{theorem}
For each $m\in \ZZ$, the family  $(L^{( m)}_n)_{n\in \NN} $ is total in $\textup{L}^2(\RR^{+},s^{\vert m \vert} e^{-s} ds)$. Moreover, the spectrum of the operator $\mathcal{T}_{m}$ is
\[\textup{Sp}(\mathcal{T}_{m}) = \left\{2k+1+\vert m \vert-m : k\in \NN \right\}\,. \]
\end{theorem}

\section{About the determinant and the trace of the Hessian}\label{app.B}
From the definition of isothermal coordinates, we have
\[g_{p^{*}}(U,V) = e^{2\eta(0)} g_0( \dd \phi_{p^{*}} U ,\dd \phi_{p^{*}} V), \qquad \text{ for all } U,V \in T_{p^{*}} M\,.\]
In the matrix expression, we have
\begin{equation}\label{three}
G_{p^{*}} = e^{2\eta(0)} (D\phi)_{p^{*}}^{T} (D \phi)_{p^{*}} \,,
\end{equation}
where $(D\phi)_{p^{*}}$ is the matrix of $d\phi_{p^{*}} $.
The relation between the Hessian of $\B$ on manifold and the Hessian of $\mathcal{B}= B\circ \phi^{-1}$ is given by
\begin{equation}\label{one}
\dd^2B_{p^{*}} (V_1,V_2) = \langle \textup{Hess} \mathcal{B}(0) \dd \phi_{p^{*}} V_1 ,\dd \phi_{p^{*}} V_2 \rangle_{\RR^2}\,. 
\end{equation}
In order to compute the trace and determinant of the Hessian at $p^{*}$, we consider the endormorphism $\mathcal{H}$ of $T_{p^{*}} M$ defined by
\begin{equation}\label{two}
(\dd^2 B)_{p^{*}}(V_1,V_2) = g_{p^{*}}(\mathcal{H} V_1,V_2) \qquad \forall V_1,V_2 \in T_{p^{*}} M \,.
\end{equation}
Additionally, \eqref{one} and \eqref{two} imply that
\begin{equation}
\langle \textup{Hess}\, \mathcal{B}(0) \dd \phi_{p^{*}} V_1 ,\dd \phi_{p^{*}} V_2 \rangle_{\RR^2} = g_{p^{*}}(\mathcal{H} V_1,V_2) \qquad \forall V_1,V_2 \in T_{p^{*}} M \,,
\end{equation}
or
\begin{equation*}
(D\phi)_{p^{*}}^{T} \,\textup{Hess}\, \mathcal{B}(0)\, (D\phi)_{p^{*}} = \mathcal{H} G_{p}\,.
\end{equation*}
Using \eqref{three}, we get
\[(D\phi)_{p^{*}}^{T} \,\textup{Hess}\, \mathcal{B}(0)\, \left[(D\phi)_{p^{*}}^{T}\right]^{-1} = e^{2\eta(0)} \mathcal{H}\,. \]
Notice that
\[ \textup{Hess}\, \mathcal{B}(0) = \begin{pmatrix}
2\alpha && 0\\
0 && 2 \gamma
\end{pmatrix}\,. \]
Let $H=\frac{1}{2}\mathcal{H}$, then we can easily compute the determinant of $H$ and the trace of $H^{\frac{1}{2}}$:
\[
\det \left(H \right)=  \frac{e^{-4 \eta(0)}}{4} \det\left(\textup{Hess}\, \mathcal{B}(0)\right) = e^{-4 \eta(0)}\alpha \gamma \,,
\]
and
\begin{align*}
\textup{Tr}\, H^{1/2} &=  \textup{Tr} \left[ \frac{e^{-\eta(0)}}{\sqrt{2}} (D\phi)_{p^{*}}^{T} \,(\textup{Hess}\, \mathcal{B}(0))^{1/2}\, \left[(D\phi)_{p^{*}}^{T}\right]^{-1}\right]\\
&=\frac{e^{-\eta(0)}}{\sqrt{2}}  \textup{Tr} \left[ (\textup{Hess}\, \mathcal{B}(0))^{1/2}\right]\\
&= e^{-\eta(0)} (\sqrt{\alpha}+\sqrt{\gamma})\,.
\end{align*}

\bibliographystyle{abbrv}
\bibliography{Ref}

\end{document}